\documentclass[12pt]{article}
\pdfoutput=1

\usepackage{amssymb,amscd,amsthm,amsmath}
\usepackage{graphicx}
\usepackage{tikz}
\usepackage{float}
\usetikzlibrary{cd}
\usetikzlibrary{intersections,calc,arrows.meta}
\usepackage{mathrsfs}

\newtheorem{theorem}{Theorem}[section]
\newtheorem{lemma}[theorem]{Lemma}
\newtheorem{proposition}[theorem]{Proposition}
\newtheorem{cor}[theorem]{Corollary}

\theoremstyle{definition}
\newtheorem{definition}[theorem]{Definition}
\newtheorem{example}[theorem]{Example}

\theoremstyle{remark}
\newtheorem{remark}[theorem]{Remark}

\theoremstyle{conjecture}
\newtheorem{conjecture}[theorem]{Conjecture}
\theoremstyle{problem}
\newtheorem{problem}[theorem]{Problem}

\numberwithin{equation}{section}

\newcommand{\C}{\mathbb{C}}

\newcommand{\K}{\mathbb{K}}

\newcommand{\N}{\mathbb{N}}
\newcommand{\Q}{\mathbb{Q}}
\newcommand{\R}{\mathbb{R}}
\newcommand{\T}{\mathbb{T}}
\newcommand{\Z}{\mathbb{Z}}

\newcommand{\cD}{\mathcal{D}}

\newcommand{\cO}{\mathcal{O}}

\newcommand{\cE}{\mathcal{E}}
\newcommand{\cF}{\mathcal{F}}

\newcommand{\cP}{\mathcal{P}}

\newcommand{\cR}{\mathcal{R}}
\newcommand{\cZ}{\mathcal{Z}}

\newcommand{\fP}{\mathfrak{P}}

\newcommand\Aut{\operatorname{Aut}}
\newcommand\Out{\operatorname{Out}}
\newcommand\Inn{\operatorname{Inn}}

\newcommand\id{\operatorname{id}}
\newcommand\Ad{\operatorname{Ad}}
\newcommand\Hom{\operatorname{Hom}}

\newcommand{\Ima}{\operatorname{Im}}

\newcommand\Ext{\operatorname{Ext}}

\newcommand{\ob}{\operatorname{ob}}
\newcommand{\tob}{\operatorname{\widetilde{ob}}}
\newcommand{\talpha}{\widetilde{\alpha}}
\newcommand{\tbeta}{\widetilde{\beta}}
\newcommand{\ev}{\mathrm{ev}}
\newcommand{\tu}{\widetilde{u}}

\newcommand{\tW}{\widetilde{W}}
\newcommand{\tZ}{\widetilde{Z}}
\newcommand{\pr}{\mathrm{pr}}
\newcommand{\tK}{\widetilde{K}}
\newcommand{\tDelta}{\tilde{\Delta}}
\newcommand{\hhalpha}{\hat{\hat{\alpha}}}
\newcommand{\bE}{\bar{E}}
\newcommand{\op}{\mathrm{op}}

\title{$G$-kernels of Kirchberg algebras}

\author{Masaki Izumi
\thanks{Supported in part by JSPS KAKENHI Grant Number JP20H01805}\\
Graduate School of Science \\
Kyoto University \\
Sakyo-ku, Kyoto 606-8502, Japan} 

\begin{document} 
\maketitle
\centerline{In memory of Eberhard Kirchberg} 
\begin{abstract} 
A $G$-kernel is a group homomorphism from a group $G$ to the outer automorphism group of a C$^*$-algebra.  
Inspired by recent work of Evington and Gir\'{o}n Pacheco in the stably finite case, we introduce a new invariant of 
a $G$-kernel using $K$-theory, and deduce several new constraints of the obstruction classes of 
$G$-kernels in the purely infinite case. 
We classify $\Z^n$-kernels for strongly self-absorbing Kirchberg algebras in the bootstrap category in terms 
of our new invariant and the Dadarlat-Pennig theory of continuous fields of strongly self-absorbing C$^*$-algebras. 
\end{abstract}
\section{Introduction} 

A $G$-kernel is a group homomorphism from a group $G$ into the outer automorphism group $\Out(A)$ of an operator algebra $A$. 
With a $G$-kernel $\alpha:G\to \Out(A)$, we can associate the third cohomology obstruction $\ob(\alpha)\in H^3(G,U(Z(A)))$ 
for $\alpha$ to lift to a cocycle $G$-action on $A$ (see \cite{Su80}), which is the most significant invariant 
for $G$-kernels. 
Indeed, in the case of a countable discrete amenable group $G$ and the hyperfinite II$_1$ factor $A=\cR$, 
this is known to be a complete invariant up to conjugacy due to Connes \cite{C75}, \cite{C77}, Jones \cite{J80}, and Ocneanu \cite{O85}
(see also \cite{KT03}, \cite{KT07}, \cite{KT09}, and \cite{Mas22} for the case of the other injective factors). 
Finite group $G$-kernels and their obstruction classes also play important roles in the conformal field theory models 
(see  \cite{EG22} for example).

Working on group actions on operator algebras, for a long time the present author expected an interesting interplay between 
K-theory and the third cohomology obstruction arising from $G$-kernels of C$^*$-algebras. 
The first result in this direction was obtained only recently by Evington and Gir\'{o}n Pacheco \cite{EGP23}, where 
they showed that $\ob(\alpha)$ is always trivial for the Jiang-Su algebra $A=\cZ$, and that a strong K-theoretical restriction 
occurs if $G$ is finite and $A$ is a UHF algebra. 
Their main technical tool is the de la Harpe-Skandalis determinant, algebraic K-theory in other words, 
which works only for stably finite C$^*$-algebras. 
Therefore it is desirable to introduce an alternative invariant by using only the topological K-theory, which is applicable 
to a wider class of C$^*$-algebras. 
One of the purposes of this paper is to accomplish this task.  

For a unital simple C$^*$-algebra $A$ whose unitary group $U(A)$ is connected, we introduce a new invariant 
$\tob(\alpha)\in H^3(G,K^\#_0(A))$ of a $G$-kernel $\alpha$, where $K^\#_0(A)$ is an extension of $\T$ by $K_0(A)$ 
for a large class of C$^*$-algebras (see Definition \ref{deftob}). 
The usual obstruction $\ob(\alpha)$ is the image of $\tob(\alpha)$ under the map induced by the surjection $K^\#_0(A)\to \T$, and hence 
$\tob(\alpha)$ carries more information than $\ob(\alpha)$. 
As we expected, our new invariant $\tob(\alpha)$ give significantly strong restriction to $\ob(\alpha)$ in the purely infinite case. 
For example, we can show by using $\tob(\alpha)$ that $\ob(\alpha)$ is trivial for the following two cases: 
(i) $G=\Z_2$ and the odd Cuntz algebras $A=\cO_{2n+1}$ (Theorem \ref{finite Cuntz}), 
and (ii) any finite $G$ and the infinite Cuntz algebra $A=\cO_\infty$ (Theorem \ref{infinite Cuntz}). 
In particular, we see that $\cZ$ and $\cO_\infty$ behave in the same way as far as finite group $G$-kernels are concerned, 
which adds a new example to many common features shared by $\cZ$ and $\cO_\infty$. 
However, when $G$ is infinite, e.g. $G=\Z^n$ with $n\geq 3$, the situation is completely different and $\cO_\infty$ 
may have non-trivial $\ob(\alpha)$. 

The recent striking work of Gabe and Szab\'o \cite{GS} on the dynamical Kirchberg-Phillips theorem shows that the theory of 
group actions on Kirchberg algebras has sufficiently matured by now, and we have enough tools to start classification 
of $G$-kernels too.  
Their result together with Meyer's work \cite{Me19} solved a conjecture raised by the author \cite{I10},\cite{IMII}, 
which allows us to use continuous fields of C$^*$-algebras to study group actions on the Kirchberg algebras. 
This typically works very well for (stabilized) strongly self-absorbing Kirchberg algebras thanks to Dadarlat-Pennig's generalized 
Dixmier-Doudary theory \cite{DP-I},\cite{DP-II}. 
Therefore it is reasonable to start classification of $G$-kernels with strongly self-absorbing Kirhberg algebras, 
which is another purpose of this paper. 
In Theorem \ref{main}, we classify $\Z^n$-kernels for the strongly self-absorbing Kirchberg algebras in the bootstrap category, 
in terms of our new invariant and the Dadarlat-Pennig theory.  

Looking back on my fledgling period, I feel very fortunate to have attended Kirchberg's famous Geneva talk in summer of 1994, 
and  his serial lectures in the Fields Institute in the subsequent winter period (his famous forever preprint \cite{Kir94} 
reminds me of the atmosphere at that time).
Without these valuable experiences, I would never have devoted my work to the study of C$^*$-algebras as seriously.  

The author would like to thank Sergio Gir\'{o}n Pacheco, Ulrich Pennig, Hiroki Matui, and Yuhei Suzuki for useful discussions, and 
the anonymous referee for careful reading. 
He would like to thank Isaac Newton Institute for its hospitality.  

\section{Preliminaries}
Throughout this paper, we assume that $A$ is a simple unital C$^*$-algebra.  
We denote by $U(A)$ the unitary group of $A$, and by $\Aut(A)$ the automorphism group of $A$.   
For $u\in U(A)$, we denote by $\Ad u$ the automorphism of $A$ defined by $\Ad u(x)=uxu^*$ for $x\in A$. 
An automorphism of $A$ is called inner if it is of the form $\Ad u$. 
We denote by $\Inn(A)$ the set of inner automorphisms, which is a normal subgroup of $\Aut(A)$. 
The outer automorphism group $\Out(A)$ of $A$ is defined by the quotient group $\Aut(A)/\Inn(A)$. 

By a trace of a C$^*$-algebra, we always mean a tracial state. 
We denote $\tK_0(A)=K_0(A)/\langle[1]_0\rangle$, where $[1]_0$ is the $K_0$-class of $1_A$. 
We denote by $\tau_*$ the homomorphism from $K_0(A)$ to $\R$ induced by a trace $\tau$. 
We denote by $\K$ the algebra of compact operators on a separable infinite dimensional Hilbert space, 
and by $A^s$ the stabilization $A\otimes \K$ of $A$. 
For an integer $n\geq 2$, we denote by $M_{n^\infty}$ the UHF algebra of type $n^\infty$. 
For a set $\fP$ of prime numbers, we denote 
$$M_{\fP^\infty}=\bigotimes_{p\in \fP}M_{p^\infty},$$
where by convention $M_\emptyset=\C$.  
If $\fP$ is the set of all prime numbers, we denote $M_\Q=M_{\fP^\infty}$. 
We denote by $\cO_n$ the Cuntz algebra, and by $\cZ$ the Jiang-Su algebra. 

We always assume that $G$ is a countable (or finite) discrete group. 
For a subgroup $H$ of $G$, the quotient map $G\to G/H$ is denoted by $q_{G\to G/H}$, or simply by $q$ 
if there is no possibility of confusion. 
For $n\in \N$, we denote $\Z_n=\Z/n\Z$. 

We denote $\T=\{z\in \C;\; |z|=1\}$, which is a multiplicative group. 
When we identify $\T$ with $\R/\Z$, we always use the surjection $\R\ni r\mapsto e^{2\pi i r}\in \T$.   

A $G$-kernel is a group homomorphism $\alpha:G\to \Out(A)$, and we always assume that $\alpha$ is \textit{injective} 
in this paper. 

We recall the definition of the obstruction class $\ob(\alpha)\in H^3(G,\T)$ of a $G$-kernel $\alpha:G\to \Out(A)$ now. 
We choose a set theoretical lifting $\talpha:G\to \Aut(A)$ of $\alpha_g$ and unitaries $u(g,h)\in U(A)$ satisfying 
$$\talpha_g\circ \talpha_h=\Ad u(g,h)\circ \talpha_{gh}.$$ 
We call such a pair $(\talpha,u)$ a lifting of $\alpha$ (called an anomalous action in \cite{J21}).  
Associativity implies 
$$\Ad\left(\talpha_g(u(h,k))u(g,hk)\right)\circ \talpha_{ghk}=\Ad(u(g,h)u(gh,k))\circ \talpha_{ghk},$$
and since the center of $A$ is trivial, there exists $\omega(g,h,k)\in \T$ satisfying 
\begin{equation}\label{obs}
\talpha_g(u(h,k))u(g,hk)=\omega(g,h,k)u(g,h)u(gh,k).
\end{equation}
We can show that $\omega$ satisfies the 3-cocycle relation. 
The definition of $\omega\in Z^3(G,\T)$ depends on the choice of the lifting $(\talpha,u)$. 
However, a different choice replaces $\talpha_g$ with $\Ad v_g\circ \talpha_g$ and $u(g,h)$ with $\mu(g,h)v_g\talpha_g(v_h)u(g,h)v_{gh}^*$,
where $v_g\in U(A)$ and $\mu(g,h)\in \T$. 
This ends up replacing $\omega(g,h,k)$ with $\partial \mu(g,h,k) \omega(g,h,k)$, and hence the cohomology class 
$[\omega]\in H^3(G,\T)$ depends only on $\alpha$. 

\begin{definition} For a $G$-kernel $\alpha$, we define the obstruction class $\ob(\alpha)\in H^3(G,\T)$ of $\alpha$ by 
the cohomology class of the cocycle $\omega$ in Eq.(\ref{obs}). 
The obstruction class $\ob(\alpha)$ depends only on the conjugacy class of $\alpha$ in $\Hom(G,\Out(A))$.  
\end{definition}
 
For a given pair of $A$ and $G$, we can ask the following two fundamental problems about $G$-kernels of $A$: 
\begin{itemize}
\item[(1)] The realization problem determining possible values of $\ob(\alpha)$. 
\item[(2)] The classification problem seeking sufficiently many invariants to distinguish $G$-kernels up to conjugacy. 
\end{itemize}
When $c\in H^3(G,\T)$ is realized as $\ob(\alpha)$ of a $G$-kernel $\alpha:G\to \Out(A)$, we say that $c$ is realized in $A$ 
for simplicity.  
We introduce several variants of the obstruction class in this paper, and we ask (1) for these variants too. 

As mentioned in the introduction, the two problems are completely solved for the hyperfinite II$_1$ factor $\cR$ and amenable $G$. 

\begin{theorem}[Connes, Jones, Ocneanu] Let $G$ be a countable amenable group, and let $\cR$ be the hyperfinite II$_1$ factor. 
Then the third cohomology obstruction is a completely invariant for $G$-kernels up to conjugacy. 
Moreover, every class in $H^3(G,\T)$ can be realized in $\cR$.  
\end{theorem}

As pointed out in \cite{EGP23}, the realization part of the above work has important implications in C$^*$-algebras too. 
For example, Connes' construction in \cite{C77} shows that every class in $H^3(\Z_n,\T)$ is realized in $M_{n^\infty}$. 

More generally, Jones' construction in \cite{J79}, \cite{J80} shows the following (see \cite[Theorem 4.3]{EGP23}): 

\begin{theorem}\label{J} For every finite group $G$, every class in $H^3(G,\T)$ is realized in $M_{|G|^\infty}$. 
\end{theorem} 

Combining Jones' construction with Kirchberg's $\cO_2$ theorem, we get the following:   

\begin{theorem}\label{JK} For every countable discrete group $G$, every class in $H^3(G,\T)$ 
is realized in the Cuntz algebra $\cO_2$. 
\end{theorem}

\begin{proof} Let $c\in H^3(G,\T)$. 
Then \cite[Lemma 2.3]{J79} shows that there exists a countable group $\tilde{G}$ with a surjective homomorphism 
$p:\tilde{G}\to G$ such that $N=\ker p$ is abelian and $p^*c=0$ in $H^3(\tilde{G},\T)$. 
We choose an outer action (say a faithful quasi-free action) $\beta$ of $\tilde{G}$ on $\cO_\infty$. 
Then \cite[Theorem 2.5]{J79} shows that the class $c$ is realized in the twisted crossed product 
$A=\cO_\infty\rtimes_{\beta,\mu} N$ with an appropriate 2-cocycle $\mu\in Z^2(N,\T)$ 
(see \cite[Section 3]{J21} and \cite[Corollary 6.2.5]{GP23} too).  
Since $\beta$ is outer and $N$ is abelian, $A$ is a Kirchberg algebra. 
Now as $c$ is realized in $A\otimes \cO_2$ too, Kirchberg's $\cO_2$ theorem finishes the proof. 
\end{proof}

We recall Evington and Gir\'{o}n Pacheco's argument in \cite{EGP23}. 
Let $\tau$ be a trace, and let $u\in U(A)_0$, where $U(A)_0$ is the connected component of $1_A$ in $U(A)$. 
Then the de la Harpe Skandalis determinant $\Delta_\tau(u)\in \R+\tau_*(K_0(A))$ is defined as follows. 
We choose a smooth path $\{\tilde{u}(t)\}_{t\in [0,1]}$ from $1$ to $u$ in $U(A)_0$, and set 
$$\Delta_\tau(u)=\frac{1}{2\pi i}\int_0^1\tau(\tilde{u}(t)^{-1}\tilde{u}'(t))dt+\tau_*(K_0(A)).$$
Then $\Delta_\tau: U(A)_0\to \R/\tau_*(K_0(A))$ is a well-defined group homomorphism (see \cite[Lemme 1, Proposition 2]{HS84}). 

Assume now that $U(A)$ is connected and $A$ has a trace $\tau$ preserved by a $G$-kernel $\alpha$. 
Let $(\talpha,u)$ be a lifting of a $G$-kernel $\alpha:G\to \Out(A)$. 
Then we have
$$\Delta_\tau(u(h,k))+\Delta_\tau(u(g,hk))=\Delta_\tau(\omega(g,h,k))+\Delta_\tau(u(g,h))+\Delta_\tau(u(gh,k)).$$
Let $s:\T\to \R/\tau_*(K_0(A))$ be the map given by $e^{2\pi i t}\mapsto t+\tau_*(K_0(A))$. 
Then the above equality means $s_*(\ob(\alpha))=0$. 
This immediately implies the following theorem. 

\begin{theorem}[{\cite[Theorem A]{EGP23}}] For the Jiang-Su algebra $\cZ$, the obstruction $\ob(\alpha)$ is always trivial 
for any  discrete group $G$ and any $G$-kernel $\alpha:G\to \Out(\cZ)$. 
\end{theorem}

Using $s_*(\ob(\alpha))=0$ together with the fact that a $G$-kernel restricts to corners by projections keeping 
the same obstruction, they also obtained the following theorem. 

\begin{theorem}[{\cite[Theorem B]{EGP23}}] Let $A$ be a UHF algebra and let $G$ be a finite group. 
Let $\alpha:G\to \Out(A)$ be a $G$-kernel, and let $r$ be the order of $\ob(\alpha)$, which is not 0. 
Then $A\cong A\otimes M_{r^\infty}$.  
\end{theorem}

The coefficient short exact sequence 
\begin{equation}\label{CSES}
0\to \tau_*(K_0(A))/\Z\to \T\to \R/\tau_*(K_0(A))\to 0
\end{equation}
induces a long exact sequence of cohomology, and $s_*(\ob(\alpha))=0$ implies that $\ob(\alpha)$ comes from a class in $H^3(G,\tau_*(K_0(A))/\Z)$, 
which was used in the second theorem. 
In fact, it is more convenient to introduce an invariant valued in the cohomology group $H^3(G,\tau_*(K_0(A))/\Z)$ directly.  

\begin{definition} Assume that $A$ has a trace $\tau$ preserved by a $G$-kernel $\alpha:G\to \Out(A)$ and $U(A)$ is connected. 
We choose a lifting $(\talpha,u)$ of $\alpha$ satisfying $u(g,h)\in \ker \Delta_\tau$. 
Then the resulting cocycle $\omega$ determined by Eq.(\ref{obs}) satisfies 
$\omega(g,h,k)\in e^{2\pi i\tau_*(K_0(A))}$. 
With such a cocycle, we define 
$$\ob_\tau(\alpha)=[\omega]\in H^3(G,\tau_*(K_0(A))/\Z).$$ 
When $A$ has a unique trace and $K_0(A)$ has no infinitesimal elements, we simply denote 
$H^3(G,\tau_*(K_0(A))/\Z)$ by $H^3(G,\tK_0(A))$ (recall $\tK_0(A)=K_0(A)/\Z[1]_0$). 
\end{definition}

To see that $\ob_\tau(\alpha)$ is well defined, let $(\talpha',u')$ be another lifting of $\alpha$ satisfying 
$\Delta_\tau(u'(g,h))=0$. 
Then there exist $v_g\in U(A)$ and $\mu(g,h)\in \T$ satisfying $\talpha'_g=\Ad v_g\circ \talpha_g$, 
$u'(g,h)=\mu(g,h)v_g\talpha_g(v_h)u(g,h)v_{gh}^*$, and 
$$\Delta_\tau(\mu(g,h))=\Delta_\tau(v_{gh})-\Delta_\tau(v_g)-\Delta_\tau(v_h).$$
This means that there exist $\eta(g)\in \R$ and $\zeta(g,h)\in \tau_*K_0(A)$ satisfying
$$\mu(g,h)=e^{2\pi i \partial \eta(g,h)}e^{2\pi i \zeta(g,h)},$$
which shows that $\ob_\tau(\alpha)$ is well-defined. 

\begin{remark}\label{injective} The long exact sequence arising from the coefficient short exact sequence Eq.(\ref{CSES}) 
shows that the map  
$$H^3(G,\tau_*K_0(A)/\Z)\to H^3(G,\T)$$
is injective if and only if the connecting map 
$$\partial_A: H^2(G,\R/\tau_*K_0(A))\to H^3(G,\tau_*K_0(A)/\Z)$$
is 0. 
The universal coefficient theorem shows that this is equivalent to 
$$\Ext(H_2(G,\Z),\tau_*K_0(A)/\Z)=\{0\}.$$ 
This condition is satisfied if either $H_2(G,\Z)$ is free e.g. $G=\Z^n$, or $\tau_*K_0(A)/\Z$ is divisible e.g. $A=M_{\fP^\infty}$.  
\end{remark}

The above two theorems inspire the following two conjectures. 

\begin{conjecture}\label{C1} 
Let $G$ be a countable discrete group, and assume that $c\in H^3(G,\T)$ has a non-zero finite 
order $r$. 
Then $c$ is realized as $\ob(\alpha)$ of a $G$-kernel $\alpha:G\to \Out(M_{r^\infty})$. 
\end{conjecture}

Conjecture \ref{C1} is true for every finite abelian group $G$ thanks to Theorem \ref{J} as it suffices to verify the conjecture 
for each prime component of $G$ (see \cite[Corollary 10.2, Theorem 10.3]{B82}). 
However, it is not clear if the conjecture holds for general finite groups. 

\begin{conjecture}\label{C2} Let $G$ be a countable discrete group with a reasonable finiteness condition for its cohomology, 
and let $\fP$ be a set of prime numbers. 
Then every class in $H^3(G,\tK_0(M_{\fP^\infty}))$ can be realized as $\ob_\tau(\alpha)$ of a $G$-kernel $\alpha:G\to \Out(M_{\fP^\infty})$.  
\end{conjecture}

In view of  
$$\tK_0(M_{\fP^\infty})\cong \bigoplus_{p\in \fP}\Z[1/p]/\Z=\bigoplus_{p\in \fP}\varinjlim_m\Z_{p^m},$$
it is reasonable to tackle  Conjecture \ref{C2} assuming a finiteness condition on $G$ such as $FP_3$ 
(see \cite[p.193]{B82} for the definition). 
Under this assumption, we have  
$$H^3(G,\tK_0(M_{\fP^\infty}))=\bigoplus_{p\in \fP}\varinjlim_{m} H^3(G,\Z_{p^m}),$$
(see \cite[Chapter VIII, (Proposition 4.6]{B82}), and Conjecture \ref{C2} is reduced to the case of $M_{p^\infty}$, 
and is further reduced to Conjecture \ref{C1} as the subtle difference between $\ob(\alpha)$ and $\ob_\tau(\alpha)$ 
can be ignored in this case (see Remark \ref{injective} above). 
Indeed, this reduction argument works for $\Z^n$. 

\begin{theorem} Conjecture \ref{C2} is true for $G=\Z^n$. 
\end{theorem}

\begin{proof} As stated above, it suffices to verify the statement for the UHF algebra $M_{p^\infty}$ for each prime $p$. 
Furthermore, it suffices to show that every class in ${\iota_m}_*H^3(\Z^n,\Z_{p^m})$ is realized in $M_{p^\infty}$, 
where $\iota_m:\Z_{p^m}\to \Z[1/p]/\Z$ is the inclusion map. 
We write $H_n(G)=H_n(G,\Z)$ and $C=\Z_{p^m}$ for simplicity. 

Let $c\in H^3(\Z^n,C)$. 
Since $H_*(\Z^n)$ is a finitely generated free abelian group, we can identify $c$ as an element in  
$\Hom(H_3(\Z^n),C)$ by the universal coefficient theorem.   
Let $q:\Z^n\to C^n$ be the quotient map. 
We show that $c$ is in the image of $q^*:H^3(C^n,C)\to H^3(\Z^n,C)$.  

The K\"{u}nneth formula implies that  
$$H_3(\Z^n)=\bigoplus_{i_1+i_2+\cdots+i_n=3,\ i_j=0,1}H_{i_1}(\Z)\otimes H_{i_2}(\Z)\otimes \cdots\otimes H_{i_n}(\Z) \cong 
\Z^{n(n-1)(n-2))/6},$$
and also that 
$$L=\bigoplus_{i_1+i_2+\cdots+i_n=3}H_{i_1}(C)\otimes H_{i_2}(C)\otimes \cdots\otimes H_{i_n}(C)$$
is a direct summand of $H_3(C^n)$. 
Note that we have $H_0(C)=\Z$ and $H_1(C)=C$.  
This shows that the image of $q_*:H_3(\Z^n)\to H_3(C^n)$ is a direct summand of $L$ isomorphic to $C^{n(n-1)(n-2)/6}$, 
and $\ker q_*=p^mH_3(\Z^n)$.  
Since the kernel of $c$ as an element in $\Hom(H_3(\Z^n),C)$ includes $p^mH_3(\Z^n)$,  
there exists $c'\in \Hom(H_3(C^n),C)$ satisfying $c=c'\circ q_*$. 
Now the universal coefficient theorem shows that there exists $c_1\in H^3(C^n,C)$ satisfying $q^*(c_1)=c$. 

Let $\iota:\Z[1/p]/\Z\to \T$ be the inclusion map. 
Thanks to Remark \ref{injective}, it suffices to show the realization of $(\iota\circ \iota_m)_*c$ as $\ob(\alpha)$ 
instead of ${\iota_m}_*c$ as $\ob_\tau(\alpha)$. 
Applying Theorem \ref{J} to $C^n$ and $(\iota\circ \iota_m)_*c_1\in H^3(C^n,\T)$, we see that $(\iota\circ \iota_m)_*c_1$ is realized as $\ob(\alpha)$ 
of a $C^n$-kernel $\alpha:C^n\to \Out(M_{p^\infty})$. 
Letting $\beta$ be the composition of $q$ and $\alpha$, and taking tensor product of $\beta$ and an outer action of $\Z^n$ on $M_{p^\infty}$, 
we get a desired $\Z^n$-kernel. 
\end{proof}
\section{A new invariant} 
In this section, we assume that $U(A)$ is connected. 
Note that we have a group homomorphism from $\pi_1(U(A))$ to $K_1(SA)=K_0(A)$ by Bott periodicity.  
For simplicity, we assume that $\pi_1(U(A))\cong K_0(A)$, which is the case, for example, if $A$ is Jiang-Su absorbing 
(see \cite[Theorem 3]{Ji97}, \cite[Theorem 4.8]{CGSTW}). 

\subsection{Invariant $\kappa^3(\alpha,u)$ for a cocycle action $(\alpha,u)$.} 
Our definition of a new invariant $\tob(\alpha)$ for a $G$-kernel $\alpha$ is a modification of that of an invariant 
$$\kappa^3(\alpha,u)\in H^3(G,K_0(A))$$
for a cocycle action $(\alpha,u)$ introduced in \cite[Section 8.1]{IMI}. 
Since we need it in the next section, we recall its definition here before introducing $\tob(\alpha)$. 

A cocycle action $(\alpha,u)$ of $G$ on $A$ is a pair of a map $\alpha:G\to \Aut(A)$ and a family of unitaries 
$u(g,h)\in U(A)$ satisfying 
$$\alpha_g\circ \alpha_h=\Ad u(g,h)\circ \alpha_{gh},$$
$$\alpha_g(u(h,k))u(g,hk)=u(g,h)u(gh,k).$$
A cocycle action $(\alpha,u)$ is outer if $\alpha_g$ is outer for every $g\in G\setminus \{1\}$.

We say that two cocycle actions $(\alpha,u)$ and $(\beta,v)$ on $A$ are equivalent if there exist $w_g\in U(A)$ 
satisfying $\beta_g=\Ad w_g\circ \alpha_g$, and 
$$v(g,h)=w_g\alpha_g(w_h)u(g,h)w_{gh}^*.$$ 
Up to equivalence, we may and do assume that $\alpha_e=\id$ and $u(g,e)=u(e,g)=1_A$. 
If there exists no non-trivial order 2 element in $G$, we may further normalize $(\alpha,u)$ so that 
$\alpha_{g^{-1}}=\alpha_g^{-1}$ and $u(g,g^{-1})=1$ hold.

\begin{definition}\label{kappa} For a cocycle $G$-action $(\alpha,u)$ on $A$, 
we choose a continuous path $\{\tu_{g,h}(t)\}_{t\in [0,1]}$ in $U(A)$ from 1 to $u_{g,h}$ 
for each pair $g,h$. 
Then 
$$\partial \tu(g,h,k)(t):=\alpha_g(\tu(h,k)(t))\tu(g,hk)(t)\tu(gh,k)(t)^{-1}\tu(g,h)(t)^{-1}$$
is a based loop in $U(A)$, and $[\partial \tu_{g,h,k}]_0\in K_0(A)$ form a 3-cocycle in $Z^3(G,K_0(A))$. 
We define $\kappa^3(\alpha,u)$ to be its cohomology class in $H^3(G,K_0(A))$. 
The class $\kappa^3(\alpha,u)$ does not depend on the choices of the paths.  
\end{definition}

If $(\alpha,u)$ and $(\beta,v)$ are equivalent, we have $\kappa^3(\alpha,u)=\kappa^3(\beta,v)$. 

We say that two cocycle actions $(\alpha,u)$ on $A$ and $(\beta,v)$ on $B$ are cocycle conjugate if 
there exists an isomorphism $\theta:A\to B$ such that $(\theta\circ\alpha\circ\theta^{-1},\theta(u))$ 
and $(\beta,v)$ are equivalent. 
Under this relation, we have 
$\theta_*\kappa^3(\alpha,u)=\kappa^3(\beta,v)$. 

\subsection{Invariant $\tob(\alpha)$ for a $G$-kernel $\alpha$.}
Now we introduce a module $K^\#_0(A)$ and an invariant $\tob(\alpha)\in H^3(G,K^\#_0(A))$ for a $G$-kernel $\alpha$. 
We state some of their basic properties without giving detailed proofs, and the reader is referred to 
Gir\'{o}n Pacheco's thesis \cite[Section 5]{GP23} for details. 

We set 
$$K^\#_0(A)=\{f\in U(C([0,1],A));\; f(0)=1,\; f(1)\in \T\}/U(C_0((0,1),A))_0,$$
where $U(C_0((0,1),A))_0$ is the connected component of 1 in  
$$\{f\in U(C([0,1],A));\; f(0)=f(1)=1\}.$$ 
Then $K^\#_0(A)$ is an abelian group. 
Let $\ev_1$ be the evaluation map at 1, and let 
$$j_A:K_0(A)=\pi_1(U(A))\to K_0^\#(A)$$
be the inclusion map. 
Then we have a short exact sequence:
\begin{equation}\label{exse}
0\to K_0(A)\xrightarrow{j_A}  K_0^\#(A)\xrightarrow{\ev_1} \T\to 0.
\end{equation}

For $r\in \R$, let $e_r(t)=e^{2\pi i rt}$. 
For two modules $M_1$ and $M_2$, we denote by $\pr_i$ for $i=1,2$ the projection from $M_1\times M_2$ 
onto the $i$-the component. 
We abuse the notation and use the same symbole for the map 
$$(M_1\times M_2)/\langle(m_1,m_2)\rangle\to M_i/\langle m_i\rangle $$
induced by $\pr_i$. 

The following presentation of $K^\#_0(A)$ is useful to identify $K^\#_0(A)$ in many concrete examples. 

\begin{lemma}\label{presentation}
\begin{itemize}
\item[$(1)$]
There exists a unique isomorphism 
$$\varphi_A:K^\#_0(A)\to (K_0(A)\times \R)/\Z([1]_0,-1),$$ 
satisfying 
$$\varphi_A\circ j_A(x)=[(x,0)],$$ 
$$\varphi_A([e_r])=[(0,r)],$$
for all $x\in K_0(A)$ and $r\in \R$. 
We also have $\mathrm{ev}_1=\pr_2\circ \varphi_A$. 
\item[$(2)$] Assume that there exists $\rho\in\Hom(K_0(A),\R)$ satisfying $\rho([1]_0)=1$. 
Then there exists a unique isomorphism $\psi_{A,\rho}:K_0^\#(A)\to \tK_0(A) \times \R$ 
satisfying 
$$\psi_{A,\rho}\circ j_A(x)=([x],\rho(x)),$$
$$\psi_{A,\rho}([e_r])=(0,r),$$
for all $x\in K_0(A)$ and $r\in \R$.
Moreover, we have 
$$\ev_1\circ \psi_{A,\rho}^{-1}([x],y)=e^{2\pi i(-\rho(x)+y)}.$$
\end{itemize}
\end{lemma}

\begin{proof} (1) The two maps $j_A:K_0(A)\to K_0^\#(A)$ and $\R\ni r\mapsto [e_r]\in K_0^\#(A)$ induce 
a surjective homomorphism from $K_0(A)\times \R$ onto $K_0^\#(A)$. 
Since its kernel is $\Z([1]_0,-1)$, we get the statement. 

(2) The statement follows from the fact that there exists an isomorphism 
$$(K_0(A)\times \R)/\Z([1]_0,-1)\ni [(x,r)]\mapsto ([x],\rho(x)+r)\in \tK_0(A)\times \R.$$
\end{proof}

\begin{definition}\label{deftob} For a $G$-kernel $\alpha:G\to \Out(A)$ with a lifting $(\talpha,u)$, 
we choose a continuous path $\{\tu(g,h)(t)\}_{t\in [0,1]}$ in $U(A)$ from 1 to $u(g,h)$ for each pair $g,h\in G$. 
Then 
$$\widetilde{\omega}(g,h,k)(t):=\talpha_g(\tu(h,k)(t))\tu(g,hk)(t)\tu(gh,k)(t)^{-1}\tu(g,h)(t)^{-1}$$
is a continuous path from 1 to $\omega(g,h,k)\in \T$ in $U(A)$, where $\omega$ is as in Eq.(\ref{obs}), 
and $[\widetilde{\omega}(g,h,k)]\in K^\#_0(A)$ form a 3-coycle. 
We define $\tob(\alpha)$ to be its cohomology class in $H^3(G,K^\#_0(A))$. 
The class $\tob(\alpha)$ does not depend on any choices made for its definition, and depends only on $\alpha$.  

We define the reduced version of $\tob(\alpha)$ by 
$$\tob^r(\alpha)=-(\pr_1\circ \varphi_A)_*\tob(\alpha)\in H^3(G,\tK_0(A)).$$ 
When $\rho$ as in Lemma \ref{presentation} is available, we have $\tob^r(\alpha)=-(\pr_1\circ \psi_{A,\rho})_*\tob(\alpha)$ too. 
\end{definition}

By construction, we have ${\ev_1}_*(\tob(\alpha))=\ob(\alpha)$, and $\tob(\alpha)$ has more information than $\ob(\alpha)$. 
When $\ob(\alpha)=0$, we can adjust $(\talpha,u)$ so that it gives a cocycle action of $G$ on $A$. 
In this case, we have ${j_A}_*(\kappa^3(\talpha,u))=\tob(\alpha)$ , which is compatible with the cohomology long exact 
sequence arising from Eq.(\ref{exse}). 

Applying Lemma \ref{presentation},(1) to the finite Cuntz algebra $A=\cO_{n+1}$, we obtain a commutative diagram 
$$
\begin{tikzcd}
K_0^\#(\cO_{n+1}) \arrow[r,"\ev_1"] \arrow[d,"\simeq"]&\T \arrow[d,equal]\\
\T\arrow[r,"n\times "]&\T
\end{tikzcd},
$$
which immediately implies 

\begin{theorem}\label{finite Cuntz} 
For every $G$-kernel $\alpha:G\to \Out(\cO_{n+1})$, its obstruction $\ob(\alpha)$ belongs to 
$nH^3(G,\T)$. 
In particular, for every prime number $p$, every $\Z_p$-kernel $\alpha:\Z_p\to \Out(\cO_{np+1})$ has trivial obstruction. 
\end{theorem}

If $p$ does not divide $n\in \N$, we have $\cO_{n+1}\cong \cO_{n+1}\otimes M_{p^\infty}$, 
and every class in $H^3(\Z_p,\T)\cong \Z_p$ is realized in $\cO_{n+1}$ by Connes' construction. 
The realization problem in the prime power order case, e.g. $A=\cO_3$ and $G=\Z_4$, is much subtler, and nothing is known about it 
to the best of the author's knowledge.  
More generally, the following problem is very fundamental. 

\begin{problem} Decide the range of $\tob$ for general $\Z_m$ and $\cO_{n+1}$. 
\end{problem}

Applying Lemma \ref{presentation},(2) to the infinite Cuntz algebra $A=\cO_\infty$, we obtain a commutative diagram 
$$
\begin{tikzcd}
K_0^\#(\cO_\infty) \arrow[r,"\ev_1"] \arrow[d,"\simeq"]&\T \arrow[d,equal]\\
\R\arrow[r,"q_{\R\to \T}"]&\T
\end{tikzcd},
$$
which implies 

\begin{theorem}\label{infinite Cuntz} For every countable discrete group $G$ and every $G$-kernel $\alpha:G\to \Out(\cO_\infty)$, 
its obstruction belongs to ${q_{\R\to \T}}_*H^3(G,\R)$. 
In particular, it is trivial for any finite $G$. 
\end{theorem}

We identify $K_0^\#(\cO_\infty)$ with $\R$ and $j_A([1]_0)\in K_0^\#(\cO_\infty)$ with $1$ in what follows.

\begin{conjecture}\label{C3} Let $G$ be a countable discrete group with a reasonable finiteness condition for its cohomology, 
and let $\fP$ be a (possibly empty) set of primes. 
Then every class in 
$$H^3(G,K_0^\#(M_{\fP^\infty}\otimes \cO_\infty))$$
is realized as $\tob(\alpha)$ of a $G$-kernel $\alpha:G\to \Out(M_{\fP^\infty}\otimes \cO_\infty)$.  
\end{conjecture}

Note that we have 
$$K_0^\#(M_{\fP^\infty}\otimes \cO_\infty)\cong \tK_0(M_{\fP^\infty})\times \R$$
thanks to Lemma \ref{presentation},(2). 
\subsection{Stably finite case} 
In this subsection, we assume that $A$ has a trace preserved by a $G$-kernel $\alpha:G\to \Out(A)$, 
and show that $\tob(\alpha)$ and its reduced form $\tob^r(\alpha)$ have the same information. 

We denote by $C^\infty_*([0,1],U(A))$ the set of smooth maps $f:[0,1]\to U(A)$ satisfying $f(0)=1$. 
We define $\tDelta_\tau:C^\infty_*([0,1],U(A))\to \R$ by 
$$\tDelta_\tau(f)=\frac{1}{2\pi i}\int_0^1\tau(f(t)^{-1}f'(t))dt.$$
Then we have $\Delta_\tau(f(1))=\tDelta_\tau(f)+\tau_*K_0(A)$ by definition, and 
$\tDelta_\tau(f)=\tau_*[f]_0$ if $f(1)=1$. 

\begin{lemma}\label{second0} If $f\in C^\infty_*([0,1],U(A))$ satisfies $f(1)\in e^{2\pi i\tau_*K_0(A)}$ and $\tDelta_\tau(f)=0$, 
we have $\psi_{A,\tau_*}([f])\in \tK_0(A)\times \{0\}$. 
\end{lemma}

\begin{proof} We take $x\in K_0(A)$ satisfying $f(1)=e^{2\pi i\tau_*x},$ and set 
$g(t)=f(t)e^{-2\pi i\tau_*x t}$. 
Then $g$ is a based loop in $U(A)$ giving an element of $K_0(A)$, and  
$$\tau_*[g]_0=\tDelta_\tau([g]_0)=\tDelta_\tau(f)-\tau_*x=-\tau_*x.$$
Since $f=ge_{\tau_*x}$, we get 
$$\psi_{A,\tau_*}(f)=(\left[[g]_0\right],\tau_*[g]_0+\tau_*x)=(\left[[g]_0\right],0),$$
showing the statement. 
\end{proof}

Let $\tau_{*,q}:\tK_0(A)\to \T$ be the map defined by 
$$\tau_{*,q}([x])=e^{2\pi i \tau_*x}.$$ 

\begin{theorem} We have 
$$(\pr_2\circ\psi_{A,\tau_*})_*\tob(\alpha)=0,$$ 
$$(\tau_{*,q})_*\tob^r(\alpha)=\ob_\tau(\alpha).$$ 
\end{theorem} 

\begin{proof} Let $(\talpha,u)$ be a lifting of a $G$-kernel satisfying $\Delta_\tau(u_{g,h})=0$  
so that $\ob_\tau(\alpha)$ is given by the cohomology class of 
$$\omega(g,h,k)=\talpha_g(u(h,k))u(g,hk)u(gh,k)^{-1}u(g,h)^{-1}.$$
We can choose a smooth path $\tu(g,h)\in C^\infty_*([0,1],U(A))$ from $1$ to $u(g,h)$ satisfying $\tDelta_\tau(\tu(g,h))=0$ 
for each pair $g,h\in G$.
With this choice of paths, the path
$$\widetilde{\omega}(g,h,k)(t)=\talpha_g(\tu(h,k)(t))\tu(g,hk)(t)\tu(gh,k)(t)^{-1}\tu(g,h)(t)^{-1}$$
satisfies $\tDelta_\tau(\widetilde{\omega}(g,h,k))=0$. 
Thus Lemma \ref{second0} and Lemma \ref{presentation} imply
$$\pr_2\circ \psi_{A,\tau_*}([\widetilde{\omega}(g,h,k)])=0,$$
$$\tau_{*,q}\circ\pr_1\circ \psi_{A,\tau_*}([\widetilde{\omega}(g,h,k)])=-\omega_{g,h,k}.$$
\end{proof}

Since $\cO_\infty$ is $KK$-equivalent to $\C$, we identify $K_0(A\otimes \cO_\infty)$ with $K_0(A)$. 
Note that although $A\otimes \cO_\infty$ has no trace, the homomorphism $\tau_*:K_0(A\otimes \cO_\infty)\to \R$ makes sense. 

\begin{lemma} Let $\alpha:G\to \Out(A)$ be a $G$-kernel with an invariant trace, and let $\beta:G\to \Out(\cO_\infty)$ 
be a $G$-kernel. 
Then we have 
$$(\psi_{A\otimes \cO_\infty,\tau_*})_*\tob(\alpha\otimes \beta)=(-\tob^r(\alpha),\tob(\beta))\in 
H^3(G,\tK_0(A))\times H^3(G,\R).$$
\end{lemma}

\begin{proof} Let $j_l:A\to A\otimes \cO_\infty$ and $j_r:\cO_\infty \to A\otimes \cO_\infty$ be the maps given by 
$x\mapsto x\otimes 1$ and $x\mapsto 1\otimes x$ respectively. 
Then the statement follows from the commutative diagram:
$$
\begin{tikzcd}
K_0^\#(A) \arrow[r,"{j_l}_*"] \arrow[d,"\psi_{A,\tau_*}"]&
K_0^\#(A\otimes \cO_\infty) \arrow[d,"\psi_{A\otimes \cO_\infty,\tau_*}"]
&K_0^\#(\cO_\infty)\arrow[l,"{j_r}_*"'] \arrow[d,"\psi_{\cO_\infty,\iota}"]\\
\tK_0(A)\times \R\arrow[r,"{j_l}_*\times \id"]
&\tK_0(A\otimes \cO_\infty)\times \R 
&\R\arrow[l,"0\times \id"']
\end{tikzcd}
$$
where $\iota:K_0(\cO_\infty)\to \R$ is the map given by $\iota([1]_0)=1$.  
\end{proof}

The lemma shows that the realization problem of the invariant $\tob$ for $A\otimes \cO_\infty$ is reduced to that for $A$ and for 
$\cO_\infty$. 
In particular, we get 

\begin{cor}\label{pi-stf} If Conjecture \ref{C2} is true for $(G,M_{\fP^\infty})$ and Conjecture \ref{C3} is true for $(G,\cO_\infty)$, 
then Conjecture \ref{C3} is true for $(G, M_{\fP^\infty}\otimes \cO_\infty)$. 
\end{cor}

Before finishing this section, we discuss the case where a $G$-kernel comes from a cocycle action. 
Assume that an outer cocycle action $(\alpha,u)$ has an invariant trace $\tau$. 
We abuse notation and we denote the $G$-kernel arising from $\alpha$ by the same symbol $\alpha$. 
Then of course we have $\ob(\alpha)=0$.  
However $\ob_\tau(\alpha)$ may no be trivial.  
By construction, we have ${q_{K_0(A)\to \tK_0(A)}}_*\kappa^3(\alpha,u)=-\tob^r(\alpha)$, and so 
$$\ob_\tau(\alpha)=-(q_{\tau_*K_0(A)\to \tau_*K_0(A)/\Z}\circ\tau_*)_*\kappa^3(\alpha,u).$$ 

\begin{proposition} Let $(\alpha,u)$ be an outer cocycle action of $G$ on $A$ with an invariant trace $\tau$. 
\begin{itemize}
\item[$(1)$]
The class $(\tau_*)_*\kappa^3(\alpha,u)\in H^3(G,\tau_*K_0(A))$ is the image of the class
$$[\Delta_\tau(u_{g,h})]\in H^2(G,\R/\tau_*K_0(A))$$ 
under the connecting map of the cohomology long exact sequence arising from the coefficients short exact sequence 
$$0\to \tau_*K_0(A)\to \R\to \R/\tau_*K_0(A)\to 0.$$
\item[$(2)$] The class $\ob_\tau([\alpha])\in H^3(G,\tau_*K_0(A)/\Z)$ is the image of the class
$$-[\Delta_\tau(u_{g,h})]\in H^2(G,\R/\tau_*K_0(A))$$ 
under the connecting map of the cohomology long exact sequence arising from the coefficients short exact sequence Eq.(\ref{CSES}). 
\end{itemize}
\end{proposition}

\begin{proof}
For each pair $g,h\in G$, we choose $\tu(g,h)\in C^\infty_*([0,1],U(A))$ satisfying $\tu(g,h)(1)=u_{g,h}$, 
and define $\partial \tu(g,h,k)$ as in the definition of $\kappa^3(\alpha,u)$. 
Let $\mu(g,h)=\tDelta_\tau(\tu(g,h))$. 
Then by definition, $(\tau_*)_*\kappa^3(\alpha,u)$ is the cohomology class given by $\tau_*\partial\tu(g,h,k)$. 
On the other hand, we have 
$$\tau_*\partial \tu(g,h,k)=\tilde{\Delta}_\tau(\partial \tu(g,h,k))=\partial \mu(g,h,k),$$
which shows (1) and (2). 
\end{proof}

\begin{remark}\label{sf} Above (1) together with the universal coefficient theorem shows  
$$(\tau_*)_*\kappa^3(\alpha,u)\in \Ext(H_2(G),\tau_*K_0(A))\subset H^3(G,\tau_*K_0(A)).$$
Note that $\Ext(H_2(G),\tau_*K_0(A))$ is quite often a small subgroup of $H^3(G,\tau_*K_0(A)).$ 
For example, it is trivial if either $H_2(G)$ is free, e.g. $G=\Z^n$, or $\tau_*K_0(A)$ is divisible, e.g. $A=M_{\Q}$. 
This is a sharp contrast from the case of Kirchberg algebras, for which we are going to show in the next section that 
the invariant $\kappa^3$ has rich range. 
\end{remark}
\section{Strongly self-absorbing Kirchberg algebras}
\subsection{Semigroup $\cF_A(G)$}
Recall that a unital separable C$^*$-algebra $A$ is strongly self-absorbing if there exists an isomorphism $\psi:A\to A\otimes A$ 
such that $\psi$ is approximately unitarily  equivalent to the map $l:A\to A\otimes A$, $l(x)=x\otimes 1_A$. 
The notion of strongly self-absorbing C$^*$-algebras was introduced in \cite{TW07}, and plays a very important role in 
the classification theory of amenable C$^*$-algebras. 
The reader is referred to \cite[Section 2]{DP-I} for their basic properties. 
Following \cite{DP-II}, we denote by $\cD_{pi}$ the class of strongly self-absorbing Kirchberg algebras in the bootstrap category. 
Throughout this section, we assume $A\in \cD_{pi}$ unless otherwise stated. 
Thus $A$ is isomorphic to either $M_{\fP^\infty}\otimes \cO_\infty$, with possibly empty $\fP$, or $\cO_2$. 
We identify $K_0(A)$ with a subring of $\R$.

Let $G$ be a countable discrete group. 
We denote by $\cF_A(G)$ the set of conjugacy classes of $G$-kernels $\alpha:G\to \Out(A)$. 
Using the fact that $A\otimes A$ is isomorphic to $A$, we can introduce a commutative semigroup structure on $\cF_A(G)$ by 
$$[\alpha]+[\beta]=[\alpha\otimes \beta].$$
Then $\tob$ induces a semigroup homomorphim from $\cF_A(G)$ into $H^3(G,K^\#_0(A))$, which we denote by the same symbol 
$\tob$ by abusing notation.   

The purpose of this section is to show that in some cases $\cF_A(G)$ is a group, and to determine its group 
structure up to extension. 
We warn the reader that $\cF_A(G)$ is not necessarily a group in general. 
For example, the semigroup $\cF_{\cO_2}(G)$ cannot be a group for finite non-trivial $G$ because a $G$-action with the Rohlin property 
absorbs all the other actions by tensor product (see \cite[Theorem 4.2]{I04Duke}). 

From now on, we assume that $G$ is infinite. 
\subsection{Semigroup $\cE_A(G)$}
Before working on $\cF_A(G)$, we need to determine the structure of its cocycle action analogue $\cE_A(G)$ first. 
We denote by $\cE_A(G)$ the set of cocycle conjugacy classes of outer cocycle actions $(\alpha,u)$ of $G$ on $A$. 
In a similar way as above, we can introduce a semigroup structure into $\cE_A(G)$ by tensor product so that  
the forgetful functor map gives a semigroup homomorphism $f:\cE_A(G)\to \cF_A(G)$. 
The purpose of this subsection is to give a topological interpretation of the invariant $\kappa^3(\alpha,u)$ defined on $\cE_A(G)$.  

Recall that we denote $A^s=A\otimes \K$. 
Let $\Aut_0(A^s)$ be the connected component of $\id$ in $\Aut(A^s)$. 
For $\gamma\in \Aut(A^s)$, we see that it is in $\Aut_0(A^s)$ if and only if $\gamma_*[1_A]_0=[1_A]_0\in K_0(A^s)$ 
(see \cite[Theorem 2.5]{DP-I}). 
If a $G$-action $\gamma$ on $A^s$ satisfies $\gamma_g\in \Aut_0(A^s)$ for every $g\in G$, we say that 
$\gamma$ is a $G$-action via $\Aut_0(A^s)$. 
We say that two such actions $\gamma_1$ and $\gamma_2$ are $KK$-trivially cocycle conjugate if 
they are coycle conjugate and the conjugation map $\theta$ can be taken from $\Aut_0(A^s)$. 
We need to generalize this notion to the case where $\gamma_1$ is a $G$-action on ${A^s}^{\otimes m}$ and $\gamma_2$ is a 
$G$-action on ${A^s}^{\otimes n}$. 
In this case, we say that $\gamma_1$ and $\gamma_2$ are KK-trivially cocycle conjugate if they are cocycle conjugate via 
a conjugation map $\theta:{A^s}^{\otimes m}\to {A^s}^{\otimes^n}$ satisfying $\theta_*[1_{A^{\otimes m}}]_0=[1_{A^{\otimes n}}]_0$ 
in $K_0({A^s}^{\otimes n})$. 
We denote by $\cE_A'(G)$ the set of $KK$-trivially cocycle conjugacy classes of outer $G$-actions via $\Aut_0(A^s)$.  
We introduce a semigroup structure into $\cE_A'(G)$ by tensor product as before. 

We first show that $\cE_A(G)$ and $\cE'_A(G)$ are naturally isomorphic. 
For this purpose, first note that we may allow outer cocycle actions $(\beta,V)$ of $G$ on $A^s$ with  $\beta_g\in \Aut_0(A^s)$ 
and $V_{g,h}\in U(M(A^s))$ for all $g,h\in G$ in the definition of $\cE_A'(G)$ because such cocycle actions are always equivalent to 
genuine actions (this essentially follows from the proof of \cite[part II, Theorem 4.1.3]{Su80} in the von Neumann algebra case). 
Thus we can define a semigroup homomorphism from $\cE_A(G)$ to $\cE'(G)$ sending $[(\alpha,u)]$ to 
$[(\alpha\otimes \id_\K,u\otimes 1)]$.  

\begin{lemma} The map $\cE_A(G)\to \cE_A'(G)$ sending $[(\alpha,u)]$ to $[(\alpha\otimes \id_\K,u\otimes 1)]$ is a semigroup 
isomorphism. 
\end{lemma}

\begin{proof}
First we show that the map is a surjection. 
Let $(\gamma,V)$ be a cocycle $G$-action on $A^s$ satisfying $\gamma_g\in \Aut_0(A^s)$. 
We choose a system of matrix units $\{E_{ij}\}_{i,j\in \N}$ in $\K$ with $E_{11}$ a minimal projection in $\K$ and 
$$\sum_{i=1}^\infty E_{ii}=1$$
converges in the strong operator topology (or the strict topology in $M(\K)$). 
For each $g\in G$, we can choose a unitary $W_g\in U(M(A^s))$ satisfying $W_g\gamma_g(1_A\otimes E_{ij})W_g^*=1_A\otimes E_{ij}$ for all $i,j$. 
Such a unitary $W_g$ exists because there exists a partial isometry $w_g\in A^s$ satisfying $w_g^*w_g=\gamma_g(1_A\otimes E_{11})$ 
and $w_gw_g^*=1\otimes E_{11}$ thanks to ${\gamma_g}_*[1_A\otimes E_{11}]_0=[1_A\otimes E_{11}]_0$, and 
$$W_g=\sum_{i=1}^\infty (1\otimes E_{i1})w_g\gamma_g(1\otimes E_{1i})$$
converges in the strict topology of $M(A^s)$. 
Then $\Ad W_g\circ \gamma_g$ leaves $1_A\otimes E_{ij}$ invariant for all $i,j$, and it is of the form $\alpha_g\otimes \id_\K$. 
Since $W_g\gamma_g(W_h)V_{g,h}W_{gh}^*$ commutes with $1_A\otimes E_{ij}$ for all $i,j$,  
it is of the form $u_{g,h}\otimes 1$. 
Thus $(\gamma,V)$ is equivalent to $(\alpha\otimes \id_\K,u\otimes 1)$. 

Next we show that the map is injection. 
Assume that $(\alpha,u)$ and $(\beta,v)$ are cocycle actions of $G$ such that $(\alpha\otimes \id_\K,u\otimes 1)$ 
and $(\beta\otimes \id_\K,v\otimes 1)$ are $KK$-trivially cocycle conjugate. 
Then there exists $\theta\in \Aut_0(A^s)$ such that $(\theta\circ (\alpha\otimes \id)\circ \theta^{-1},\theta(u\otimes 1))$ and 
$(\beta\otimes \id_\K,v\otimes 1)$ are equivalent. 
As above we may assume that $\theta$ is of the form $\theta=\theta_0\otimes \id_\K$ by perturbing $\theta$ by an 
inner automorphism. 
Thus we see that $(\theta_0\circ \alpha\circ \theta_0^{-1}\otimes \id_\K,\theta_0(u)\otimes 1)$ and 
$(\beta\otimes \id_\K,v\otimes 1)$ are equivalent, and there exist unitaries $W_g\in U(M(A^s))$ satisfying 
$$\Ad W_g\circ (\beta_g\otimes \id_\K)=\theta_0\circ \alpha_g\circ \theta_0^{-1}\otimes \id_\K,$$
$$W_g(\beta_g\otimes \id_\K)(W_h)(v(g,h)\otimes 1)W_{gh}^*=\theta_0(u(g,h))\otimes 1.$$
The first equation implies that $W_g$ commutes with $1_A\otimes \K$ and it is of the form $W_g=w_g\otimes 1$. 
Thus $(\alpha,u)$ and $(\beta,v)$ are cocycle conjugate. 
\end{proof}

For our purpose, it is more convenient to have an explicit formula of a genuine action equivalent to 
$(\alpha\otimes \id_\K,u\otimes 1)$, which is given by the second dual action in the Takesaki-Takai type duality 
for cocycle actions (see \cite{PR89}). 
Although traditionally the second dual action is described in terms of the right regular representation of $G$,  
here we give an action inner conjugate to it using the left regular representation $\lambda$ 
in order to simplify the notation in the proof of Theorem \ref{primary}.  
Let $\{E_{g,h}\}_{g,h}$ be the canonical system of matrix units in $\K(\ell^2(G))$. 
We set 
\begin{equation}\label{sd}
V_g=(\sum_{s\in G}\alpha_{s^{-1}}^{-1}(u(s^{-1},g)^{-1})\otimes E_{s,s})(1_A\otimes \lambda_g)\in U(M(A\otimes \K(\ell^2(G)))).
\end{equation}
Then 
$$V_g(\alpha_g\otimes \id)(V_h)(u(g,h)\otimes 1)V_{gh}^{-1}=1,$$
and we get a genuine action of $G$ on $A^s$ given by 
$$\hhalpha_g=\Ad V_g\circ (\alpha_g\otimes\id_\K).$$

For a $G$-action $\alpha$ on $A^s$, we define a principal $\Aut(A^s)$-bundle $\cP_\alpha$ over the classifying space $BG$ by 
$$\cP_\alpha=(EG\times \Aut(A^s))/G,$$
where the $G$ action above is given by $g\cdot(x,\gamma)=(g\cdot x,\beta_g\circ \gamma)$.  
If moreover $\alpha$ is via $\Aut_0(A^s)$, we define a principal $\Aut_0(A^s)$-bundle $\cP_\alpha^{(0)}$ over $BG$ by 
$$\cP_\alpha^{(0)}=(EG\times \Aut_0(A^s))/G.$$

Since the first non-trivial homotopy group of $\Aut(A^s)$ is $\pi_0(\Aut(A^s))\cong K_0(A)^\times$ (see \cite[Theorem 2.18]{DP-I}), 
the primary obstruction to a continuous section of $\cP_\alpha\to BG$ is in 
$$H^1(BG,\pi_0(\Aut(A^s)))\cong \Hom(G,\pi_0(\Aut(A^s))),$$
which is naturally identified with the composition of $\alpha$ with the quotiont map from $\Aut(A^s)$ to $\pi_0(\Aut(A^s))$. 

Since the first non-trivial homotopy group of $\Aut_0(A^s)$ is 
$\pi_2(\Aut_0(A^s))\cong K_0(A)$, the primary obstruction to a continuous section of 
$\cP_\alpha^{(0)}\to BG$ is in $$H^3(BG,\pi_2(\Aut_0(A^s)))\cong H^3(G,K_0(A)).$$

To compute the primary obstruction in terms of a cocycle action, 
we adopt Milgram's geometric bar construction \cite{Mi67} as a model of $EG$. 
Let $\Delta^n$ be the geometric $n$-simplex 
$$\Delta^n=\{(t_0,t_1,\cdots,t_n)\in \R^{n+1};\; \sum_{i=0}^nt_i=1,\;t_i\geq 0\}.$$
We define $d^i:\Delta^{n-1}\to \Delta^n$ for $0\leq i\leq n$, and $s_i:\Delta^{n+1}\to \Delta^n$ 
for $0\leq i\leq n$ by 
$$d^i(t_0,\cdots,t_{n-1})=(t_0,\cdots,t_{i-1},0,t_{i+1},\cdots,t_{n-1}),$$
$$s^i(t_0,\cdots,t_{n+1})=(t_0,\cdots,t_{i-1},t_i+t_{i+1},t_{i+2},\cdots,t_{n+1}).$$
Then 
$$EG=(\coprod_{k=0}^\infty G\times \Delta^k\times G^k)/\sim,$$
where the equivalence relation $\sim$ is generated by 
$$(g_0;d^i(t);g_1,\cdots,g_n) 
 \sim \left\{
\begin{array}{ll}
(g_0g_1;t;g_2,\cdots,g_n) , &\quad i=0 \\
(g_0;t;g_1,\cdots,g_ig_{i+1},\cdots,g_n) , &\quad 1\leq i\leq n-1 \\
(g_0;t;g_1,\cdots,g_{n-1}) , &\quad i=n
\end{array}
\right.$$
$$(g_0;t;g_1,\cdots,g_{i-1},e,g_{i+1},\cdots,g_n)\sim (g_0;s^i(t);g_1,\cdots,g_{i-1},g_{i+1},\cdots,g_n),$$
and a $G$-action is given by $g\cdot(g_0;t;g_1,\cdots,g_n)=(gg_0;t;g_1,\cdots,g_n)$. 
The $n$-skeleton on $EG$ is 
$$E_nG=(\coprod_{k=0}^n G\times \Delta^k\times G^k)/\sim,$$
and we set $B_nG=E_nG/G$. 

We can regard $(g_0:\Delta^n:g_1,g_2,\cdots,g_n)$ as a $n$-simplex whose vertices are labeled by 
$$(g_0,g_0g_1,g_0g_1g_2,\cdots, g_0g_1\cdots g_n)$$
if $g_i\neq e$ for $i=1,2,\cdots,n$. 
\begin{figure}[H]
\centering
\begin{tikzpicture}
\draw (0,0)--(2,0)--(2.6,1.2)--(1,1.7)--(2,0);
\draw (0,0)--(1,1.7);
\draw[densely dotted](0,0)--(2.6,1.2);
\draw(0,0)node[below]{$e$};
\draw(2,0)node[below]{$g_1$};
\draw(2.6,1.2)node[right]{$g_1g_2g_3$};
\draw(1,1.7)node[above]{$g_1g_2$};
\end{tikzpicture}
\caption{$(e:\Delta^3:g_1,g_2,g_3)$}
\end{figure}
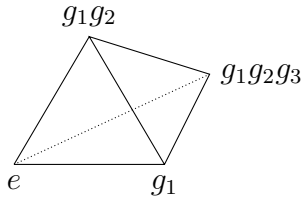

\begin{theorem}\label{primary}  Let $A$ be a strongly self-absorbing C$^*$-algebra with trivial $K_1(A)$ 
and let $G$ be a countably infinite discrete group. 
Let $(\alpha,u)$ be a cocycle action of $G$ on $A$. 
Then the primary obstruction to a continuous section of $\cP_{\hhalpha}^{(0)}\to BG$ is $\kappa^3(\alpha,u)$. 
(We do not assume that $A$ is either a Kirchberg algebra or in the bootstrap category. 
We do not assume that the cocycle action $(\alpha,u)$ is outer.)
\end{theorem}

\begin{proof}
Note that a partial section of the fiber-bundle $\cP_{\hhalpha}^0\to BG$ defined on $B_nG$ is identified with  a continuous map 
$\varphi:E_nG\to \Aut_0(A^s)$ satisfying $\varphi(g\cdot x)=\hhalpha_g\circ\varphi(x)$ because every partial section is of the form 
$x\mapsto [(x,\varphi(x))]$. 
To compute the primary obstruction class to a section of $\cP_{\hhalpha}^0\to BG$,  
our task is to choose a partial section $\varphi:E_2G\to \Aut_0(A^s)$, and comupute the element in $\pi_2(\Aut(A^s))$ arising from  
the restriction of $\varphi$ to $\partial (e:\Delta^3:g_1,g_2,g_3)$ (see \cite[Section 7.6]{DK01}). 

Let $Pr(A^s)_{[1_A]_0}$ be the set of projections in $A^s$ whose $K_0$-classes are equal to $[1_A]_0$.  
Since the map 
$$\Aut_0(A^s)\ni \gamma\to \gamma(1\otimes E_{e,e})\in Pr(A^s)_{[1_A]_0}$$
gives a homotopy equivalence between $\Aut_0(A^s)$ and $Pr(A^s)_{[1_A]_0}$ (see \cite[Corollary 2.8]{DP-I}), 
we may and do consider a $G$-equivariant map 
from $E_2G$ to $Pr(A^s)_{[1_A]_0}$ instead of a partial section $E_2G\to \Aut_0(A^s)$. 
Here the $G$-action on $Pr(A^s)_{[1_A]}$ is given by $\hhalpha$. 

Since $G$ is torsion-free, we may and do normalize $(\alpha_g,u)$ so that $\alpha_e=\id$, $\alpha_g^{-1}=\alpha_g$, 
and $u(e,g)=u(g,e)=u(g,g^{-1})=1$ hold. 
Thus we have $\alpha_{s^{-1}}^{-1}(u(s^{-1},g)^{-1})=u(s,s^{-1}g)$. 
Let 
$$W_g=\sum_{s\in G}u(s,s^{-1}g)\otimes E_{s,s}.$$
Then $\hhalpha_g=\Ad W_g\circ (\alpha_g\otimes \Ad\lambda_g)$. 
We often use the equality $u(g,h)=u(gh,h^{-1})^*$. 

For $\eta\in \ell^2(G)\setminus \{0\}$, we denote by $P_\eta\in \K(\ell^2(G))$ the projection onto $\C\eta$. 
Let $\{\delta_g\}_{g\in G}$ be the canonical orthonormal basis of $\ell^2(G)$. 
Then $E_{g,g}=P_{\delta_g}$. 

Now we construct a $G$-equivariant map $\varphi:E_2G\to Pr(A^s)_{[1_A]_0}$. 
First we define $\varphi((g:\Delta^0))=1_A\otimes E_{g,g}$, which is an equivariant map from $E_0G$ to $Pr(A^s)_{[1_A]_0}$. 
Next we extend $\varphi$ to $(e:\Delta^1:g)$ by 
$$\varphi((e:(t_0,t_1):g))=1_A\otimes P_{t_0 \delta_e+t_1\delta_{g}}.$$
To make an equivariant extension of $\varphi$ to $E_1G$, we set 
$$\varphi((g_0:(t_0,t_1):g_1))=\hhalpha_{g_0}(\varphi((e:(t_0,t_1):g_1)))
=\Ad W_{g_0}(1\otimes P_{t_0\delta_{g_0}+t_1\delta_{g_0g_1}}).$$

Now we extend $\varphi$ to $(e:\Delta^2:g_1,g_2)$. 
On the boundary of $(e:\Delta^2:g_1,g_2)$, we already have 
$$\varphi((e:(t_0,t_1,0):g_1,g_2))=1_A\otimes P_{t_0\delta_e+t_1\delta_{g_1}},$$
$$\varphi((e:(t_0,0,t_2):g_1,g_2))=1_A\otimes P_{t_0\delta_e+t_2\delta_{g_1g_2}},$$
$$\varphi((e:(0,t_1,t_2):g_1,g_2))=W_{g_1}(1_A\otimes P_{t_1\delta_{g_1}+t_2\delta_{g_1g_2}})W_{g_1}^*.$$
Let 
$$\psi_{g_1,g_2}((e:(t_0,t_1,t_2):g_1,g_2))=1_A\otimes P_{t_0\delta_e+t_1\delta_{g_1}+t_2\delta_{g_1g_2}},$$
which is a continuous map from $(e:\Delta^2:g_1,g_2)$ to $Pr(A^s)_{[1_A]_0}$. 
We need to deform $\psi_{g_1,g_2}$ on a neighborhood of $(g_1:\Delta^1:g_2)$ to extend $\varphi$. 
We choose a continuous path $\{\tu(g_1,g_2)(t)\}_{[0,1]}$ from 1 to $u(g_1,g_2)$ in $U(A)$, and set 
$$\tW_{g}(t)=\sum_{s\in G}\tu(s,s^{-1}g)(t)\otimes E_{s,s}.$$ 
We may and do assume $\tu(g,h)=\tu(gh,h^{-1})^*$.

We choose a smooth convex curve $C$ connecting $g_1$ and $g_1g_2$ inside $(e:\Delta^2:g_1,g_2)$ as in Figure 2. 
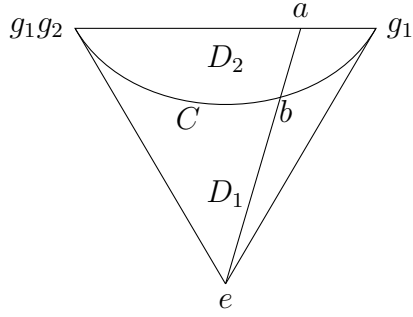
\begin{figure}[H]
\centering
\begin{tikzpicture}
\draw (0,0)--(2,3.4)--(-2,3.4)--(0,0);
\draw(-2,3.4)to[out=-60,in=240](2,3.4);
\draw(0,0)--(1,3.4);
\draw(0,0)node[below]{$e$};
\draw(2,3.4)node[right]{$g_1$};
\draw(-2,3.4)node[left]{$g_1g_2$};
\draw(1,3.4)node[above]{$a$};
\draw(0.8,2.3)node{$b$};
\draw(-0.5,2.5)node[below]{$C$};
\draw(0,3)node{$D_2$};
\draw(0,1.2)node{$D_1$};
\end{tikzpicture}
\caption{$(e:\Delta^2:g_1,g_2)$}
\end{figure}
Let $a\in (g_1:\Delta^1:g_2)$ and let $b$ be the intersection of $C$ and the line segment $ea$.  
We extend $\varphi$ to the region $D_1$ (including the boundary) below $C$ and the region $D_2$ above $C$ separately. 
We first take a homeomorphism $h_1:D_1 \to (e:\Delta^2:g_1,g_2)$ whose restriction to the line segment 
$eb$ is an affine map from $eb$ to $ea$. 
Then we set $\varphi((e:t:g_1,g_2))=\psi_{g_1,g_2}\circ h_1(t)$ for $t\in D_1$. 

For $D_2$, we define $\varphi$ as follows. 
Let $c$ be the point internally divides the line segment $ba$ in  the ratio of $r:1-r$. 
Then we set 
$\varphi(c)=\Ad\tW_{g_1}(r)(\psi_{g_1,g_2}(a))$. 
In other words, we parametrize $D_2$ by $(t,r)\in [0,1]\times [0,1]$ with each of $\{0\}\times [0,1]$ and $\{1\}\times [0,1]$ 
collapsed to one point respectively, and define $\varphi$ by 
$$\Ad\tW_{g_1}(r)(1_A\otimes P_{(1-t)\delta_{g_1}+t\delta_{g_1g_2}}).$$
Then we get a continuous extension of $\varphi$ to $(e:\Delta^2:g_1,g_2)$.   
We extend $\varphi$ to the whole $E_2G$ by setting 
$$\varphi((g_0:t:g_1,g_2))=\hhalpha_{g_0}(\varphi((e:t:g_1,g_2))).$$

Our task is to compute the element of $\pi_2(Pr(A^s)_{[1_A]_0})$ determined by the restriction of $\varphi$ to 
$\partial (e:\Delta^3:g_1,g_2,g_3)$.  
For this, it suffices to compute 
$$[\varphi|_{\partial (e:\Delta^3:g_1,g_2,g_3)}]_0-[1_A]_0\in K_0(S^2A^s)\cong K_0(A),$$
by Bott periodicity. 

We first deform $\varphi$ on $(e:\Delta^2:g_1g_2,g_3)\cup (g_1:\Delta^2:g_2,g_3)$. 
We choose a homeomorphism $h_2$ of it as in Figure 3 leaving the boundary invariant such that  
$h_1(D_i)=D_i'$, $i=3,4,5$. 
\begin{figure}[H]
\centering
\begin{tikzpicture}
\draw (-4,0)--(-2,0)--(-3,1.8)--(-4,0);
\draw (-4,0)--(-3,-1.7)--(-2,0);
\draw (2,0)--(4,0)--(3,1.7)--(2,0);
\draw (2,0)--(3,-1.7)--(4,0);
\draw (-4,0)to[out=-30,in=180](-3,-0.5);
\draw (-3,-0.5)to[out=0,in=210](-2,0);
\draw (-4,0)to[out=-50,in=180](-3,-1);
\draw (-3,-1)to[out=0,in=230](-2,0);
\draw(2,0)to[out=-60,in=240](4,0);
\draw(2,0)to[out=60,in=120](4,0);
\draw(-3,-1.7)node[below]{$e$};
\draw(-2,0)node[right]{$g_1g_2$};
\draw(-4,0)node[left]{$g_1g_2g_3$};
\draw(-3,1.7)node[above]{$g_1$};
\draw(-3,0.1)node[below]{$D_5$};
\draw(-3,-0.4)node[below]{$D_4$};
\draw(-3,-0.9)node[below]{$D_3$};
\draw(3,-1.3)node[above]{$D_3'$};
\draw(3,0.1)node[below]{$D_4'$};
\draw(3,-0.1)node[above]{$D_5'$};
\draw(3,-1.7)node[below]{$e$};
\draw(4,0)node[right]{$g_1g_2$};
\draw(2,0)node[left]{$g_1g_2g_3$};
\draw(3,1.7)node[above]{$g_1$};
\draw(0,0)node{$\to$};
\draw(-7,0)node{$h_2:$};
\end{tikzpicture}
\caption{}
\end{figure}
We can deform $\varphi$ into $\varphi_1=\varphi\circ h_2$ on $(e:\Delta^2:g_1g_2,g_3)\cup (g_1:\Delta^2:g_2,g_3)$ so that we have the 
following description of $\varphi_1$ on each region. 
On $(g_1:\Delta^2:g_2,g_3)$ we have $\varphi_1=\hhalpha_{g_1}\circ \psi_{g_2,g_3}$. 
On $D_3$, there exists a homeomorphism $h_3:D_3\to (e:\Delta^2:g_1g_2,g_3)$ satisfying $\varphi_1|_{D_3}=\psi_{g_1g_2,g_3}\circ h_3$. 
On $D_4$, the map $\varphi_1$ is described by 
$$\Ad \tW_{g_1g_2}(r)(1_A\otimes P_{(1-t)\delta_{g_1g_2}+t\delta_{g_1g_2g_3}}),$$
as in the case of $D_2$.  
We have similar description of $\varphi_1$ on $D_5$. 
We put $\varphi_1=\varphi$ on $(e:\Delta^2:g_1,g_2)\cup (e:\Delta^2:g_1,g_2g_3)$. 

\begin{figure}[H]
\centering
\begin{tikzpicture}
\draw (0,-4)--(3.5,2)--(-3.5,2)--(0,-4)--(0,0)--(3.5,2);
\draw (0,0)--(-3.5,2);
\draw (-3.5,2)to[out=-55,in=210](0,0);
\draw (0,0)to[out=-30,in=235](3.5,2);
\draw(0,-4)node[below]{$e$};
\draw(-3.5,2)node[left]{$g_1g_2$};
\draw(3.5,2)node[right]{$g_1g_2g_3$};
\draw(0,0)node[above]{$g_1$};
\draw(-0.7,-1.5)node{$D_1$};
\draw(0.7,-1.5)node{$D_6$};
\draw(-1.2,0.3)node{$D_2$};
\draw(1.2,0.3)node{$D_7$};
\end{tikzpicture}
\caption{}
\end{figure}
Secondly we deform $\varphi_1$ on 
$$(e:\Delta^2:g_1,g_2)\cup (e:\Delta^2:g_1,g_2g_3)\cup (g_1:\Delta^2:g_2,g_3).$$ 
Recall that $\varphi_1$ on $D_1$ and $D_6$ as in Figure 4 are compositions of suitable homeomorphisms and $\psi_{g_1,g_2}$ 
and $\psi_{g_1,g_2g_3}$ respectively, and $\varphi_1$ on $D_2$ and $D_7$ are described by 
$$\Ad \tW_{g_1}(r)(1_A\otimes P_{(1-t)\delta_{g_1}+t\delta_{g_1g_2}}),$$
$$\Ad \tW_{g_1}(r)(1_A\otimes P_{(1-t)\delta_{g_1}+t\delta_{g_1g_2g_3}}).$$
On $(g_1:\Delta^2:g_2,g_3)$, we have $\varphi_1=\hhalpha_{g_1}\circ \psi_{g_2,g_3}$. 
We can deform $\varphi_1$ in into $\varphi_2$ so that the same description is possible for $D_1'$, $D_2'$, $D_6'$, $D_7'$ 
as in Figure 5.  
\begin{figure}[H]
\centering
\begin{tikzpicture}
\draw (0,-4)--(3.5,2)--(-3.5,2)--(0,-4)--(0,0)--(3.5,2);
\draw (0,0)--(-3.5,2);
\draw (-3.5,2)to[out=-55,in=180](0,-0.5);
\draw (0,-0.5)to[out=0,in=235](3.5,2);
\draw(0,-4)node[below]{$e$};
\draw(-3.5,2)node[left]{$g_1g_2$};
\draw(3.5,2)node[right]{$g_1g_2g_3$};
\draw(0,0)node[above]{$g_1$};
\draw(-0.7,-1.5)node{$D_1'$};
\draw(0.7,-1.5)node{$D_6'$};
\draw(-1.2,0.3)node{$D_2'$};
\draw(1.2,0.3)node{$D_7'$};
\end{tikzpicture}
\caption{}
\end{figure}
We further deform $\varphi_2$ into $\varphi_3$ by applying the following deformation: 
$$P_{t_0\delta_e+t_1\delta_{g_1}+t_2\delta_{g_1g_2}+t_3\delta_{g_1g_2g_3}}\mapsto 
P_{t_0\delta_e+(1-s)t_1\delta_{g_1}+(t_2+\frac{st_1}{2})\delta_{g_1g_2}+(t_3+\frac{st_1}{2})\delta_{g_1g_2g_3}},$$
where $0\leq s\leq 1$ is a deformation parameter. 
Note that this does not deform $\varphi_2$ on the boundary. 

We further deform $\varphi_3$ into $\varphi_4$ so that $\varphi_4$ is described as follows.  
There exists a homeomorphism $h_3: D_8\to (e:\Delta^2:g_1g_2,g_3)$ such that $\varphi_4=\psi_{g_1g_2,g_3}\circ h_3$ on $D_8$.  
On $D_9$, the map $\varphi_4$ is described by 
$$\Ad \tW_{g_1}(r)(1_A\otimes P_{(1-t)\delta_{g_1g_2}+t\delta_{g_1g_2g_3}}).$$
We put $\varphi_4=\varphi_1$ on $(e:\Delta^2:g_1g_2,g_3)$. 
 \begin{figure}[H]
\centering
\begin{tikzpicture}
\draw (0,-4)--(3.5,2)--(-3.5,2)--(0,-4);
\draw[dashed](0,0)--(0,-4);
\draw[dashed](-3.5,2)--(0,0)--(3.5,2);
\draw (-3.5,2)to[out=-30,in=210](3.5,2);
\draw(0,-4)node[below]{$e$};
\draw(-3.5,2)node[left]{$g_1g_2$};
\draw(3.5,2)node[right]{$g_1g_2g_3$};
\draw(0,0)node[above]{$g_1$};
\draw(-0.5,-1.5)node{$D_8$};
\draw(0,1.5)node{$D_9$};
\end{tikzpicture}
\caption{}
\end{figure}

Now we deform $\varphi_4$ on $D_9\cup D_5\cup D_4$. 
In the following argument, homotopy of unitaries is understood after they are cut by the projection 
$1_A\otimes P_{\delta_{g_1g_2}}+1_A\otimes P_{\delta_{g_1g_2g_3}}$, and no issue of the topology of $U(M(A^s))$ occurs. 
On $D_9$, the projection path $p(t)=1_A\otimes P_{(1-t)\delta_{g_1g_2}+t\delta_{g_1g_2g_3}}$ is deformed as $\Ad\tW_{g_1}(r)(p(t))$. 
On $D_5$, the projection path $W_{g_1}p(t)W_{g_1}^*$ is deformed as 
$$\Ad (W_{g_1}(\alpha_{g_1}\otimes \Ad (\lambda_{g_1}))(\tW_{g_2})(r))(p(t)).$$ 
Note that the concatenation of the two unitary paths $\{\tW_{g_1}(r)\}_{r\in [0,1]}$ and 
$$\{W_{g_1}(\alpha_{g_1}\otimes \Ad (\lambda_{g_1}))(\tW_{g_2})(r)\}_{r\in [0,1]}$$
is homotopic to 
$$\{\tW_{g_1}(r)(\alpha_{g_1}\otimes \Ad (\lambda_{g_1}))(\tW_{g_2})(r))\}_{r\in [0,1]},$$ and its endpoint is 
$W_{g_1g_2}(u_{g_1,g_2}^*\otimes 1)$. 
On $D_4$, the projection path $p(t)$ is deformed as $\Ad \tW_{g_1g_2}(r)(p(t))$ (in the reversed direction), but we may replace it 
with 
$$\Ad (\tW_{g_1g_2}(r)(\tu_{g_1,g_2}(r)^*\otimes 1))(p(t)).$$ 
Now concatenation of the previous unitary path with the unitary path
$$\{\tW_{g_1g_2}(1-r)(\tu_{g_1,g_2}(1-r)^*\otimes 1)\}_{r\in [0,1]}$$
is homotopic to the unitary loop 
$$\{(\tu_{g_1,g_2}(r)\otimes 1)\tW_{g_1g_2}(r)^*\tW_{g_1}(r)(\alpha_{g_1}\otimes \Ad (\lambda_{g_1}))(\tW_{g_2})(r))\}_{r\in [0,1]}.$$
Cutting this by the projection $1_A\otimes P_{\delta_{g_1g_2}}+1_A\otimes P_{\delta_{g_1g_2g_3}}$, we get 
$$1_A\otimes P_{\delta_{g_1g_2}}+\partial\tu(g_1,g_2,g_3)(r)^*\otimes P_{\delta_{g_1g_2g_3}}.$$
Now we deform $\varphi_4$ into $\varphi_5$ so that its restriction on to $D_9\cup D_5\cup D_4$ is given by 
$$\Ad\left(1_A\otimes P_{\delta_{g_1g_2}}+(\partial\tu(g_1,g_2,g_3)(r)^*\otimes P_{\delta_{g_1g_2g_3}}\right)(p(t)),$$ 
and leave $\varphi_4|_{D_3\cup D_8}$ undeformed. 
Since $\varphi_5$ on $D_3$ and on $D_8$ are essentially the same, the homotopy group element we are looking for 
is determined by the restriction of $\varphi_5$ to $D_9\cup D_5\cup D_4$ by identify the lower side of $D_9$ and the lower side of $D_4$.   

From now on we identify $\delta_{g_1g_2}$ and $\delta_{g_1g_2g_3}$ with the canonical basis $e_1$, $e_2$ of $\C^2$, 
and identify the corner of $A^s$ by the projection $1_A\otimes P_{e_1}+1_A\otimes P_{e_2}$ with $M_2(A)$. 
Then the above computation shows that $\varphi_5$ restricted to $D_8\cup D_5\cup D_4$ is given by 
$$\left(
\begin{array}{cc}
c(t)^2 &c(t)s(t) \tu(g_1.g_2,g_3)(r)  \\
c(t)s(t)\tu(g_1.g_2,g_3)(r)^* &s(t)^2 
\end{array}
\right),
$$
where $c(t)=\frac{1-t}{\sqrt{(1-t)^2+t^2}}$, $s(t)=\frac{t}{\sqrt{(1-t)^2+t^2}}$. 
Let 
$$R(t)=\left(
\begin{array}{cc}
c(t) &-s(t)  \\
s(t) &c(t) 
\end{array}
\right), 
$$
$$Z(r)=\left(
\begin{array}{cc}
\tu(g_1.g_2,g_3)(r)  &0  \\
0 &1_A 
\end{array}
\right).
$$
We have to decide the homotopy group element given by 
$$\Phi(t,r)=Z(t)R(t)\left(
\begin{array}{cc}
1_A &0  \\
0 & 0
\end{array}
\right)R(t)^*Z(r)^*. 
$$ 
Since $\Phi(0,r)=1_A\otimes P_{e_1}$ and $\Phi(1,r)=1_A\otimes P_{e_2}$, 
it is more convenient for us to extend $\Phi$ to $[0,2]\times [0,1]$ 
by setting 
$$\Phi(t,r)=R(t-1)^*(1_A\otimes P_{e_2})R(t-1),$$
for $1\leq t\leq 2$.  
This map is homotopic to 
$$\Psi(t,r)=\left\{
\begin{array}{ll}
\tZ(t,r)(1_A\otimes P_{e_1})\tZ(t,r) , &\quad 0\leq t\leq 1 \\
(1_A\otimes P_{e_1}), &\quad 1\leq t\leq 2,
\end{array}
\right.
$$
where 
$$\tZ(t,r)=R(t)^*\left(
\begin{array}{cc}
Z(r) &0  \\
0 &1 
\end{array}
\right)R(t)
\left(
\begin{array}{cc}
1 &0  \\
0 &Z(r)^* 
\end{array}
\right).
$$
Note that $\{\tZ(t,\cdot)\}_{t\in [0,1]}$ is a unitary path in $M_2(A)$ from $\left(
\begin{array}{cc}
Z &0  \\
0 &Z^* 
\end{array}
\right)$ to $1_{M_2(A)}$. 
Comparing the orientation of our $[0,2]\times [0,1]/\sim$ with the usual isomorphism $\theta_{SA}:K_1(SA)\to K_0(A)$ 
(see \cite[Theorem 8.2.2]{Bl}), we finish the proof. 
\end{proof}

\begin{remark} Theorem \ref{primary} together with Remark \ref{sf} shows that Meyer's result \cite[Theorem 3.10]{Me19} 
does not hold for $M_{{\fP^\infty}}$ or $\cZ$.  
\end{remark}

\subsection{Relationship with the Dadarlat-Pennig theory}
From now on, we assume that $G$ is amenable and there exists a finite CW-complex model of the classifying space $BG$. 
We show under this condition that the semigroup $\cE_A'(G)$ is naturally isomorphic to Dadarlat-Pennig's group 
$\bE_A^1(BG)$, and in particular $\cE_A(G)$ is a group. 

Let $X$ be a compact metric space. 
Dadarlat-Pennig's cohomology $E_A^1(X)$ is the set of isomorphism classes of locally trivial continuous fields of 
$A^s$ over $X$, which is a group under tensor product over $C(X)$. 
Equivalently, it can be defined to be the set of isomorphism classes of principal $\Aut(A^s)$-bundles over $X$, 
which is further identified with the homotopy set $[X,B\Aut(A^s)]$. 
The reduced version $\bE_A^1(X)$ is defined by the set of isomorphism classes of principal $\Aut_0(A^s)$-bundles over $X$, 
or equivalently by $[X,B\Aut_0(A^s)]$. 

Let $\cP$ be a principal $\Aut_0(A^s)$-bundle over $X$, and assume that $X$ is connected.  
Then $\cP\times_{\Aut_0(A^s)}\Aut(A^s)$ is a principal $\Aut(A^s)$-bundle, which gives a natural map from $\bE_A^1(X)$ 
to $E^1_A(X)$. 
It was shown in \cite[Proposition 2.3]{DP-II} that this map is injective and its image is the kernel of 
the primary obstruction map 
$$\delta_0:E_A^1(X)\to H^1(X,\pi_0(\Aut(A^s)))=H^1(X,K_0(A)^\times).$$
In particular, we can regard $\bE_A^1(X)$ as a subgroup of $E_A^1(X)$ inheriting a group structure from $E_A^1(X)$. 

Let $\alpha$ be a $G$-action on $A^s$, and let $\cP_\alpha$ be the corresponding principal $\Aut(A^s)$-bundle over $BG$. 
The section algebra of the associated $A^s$-bundle ${\cP_\alpha}\times_{\Aut(A^s)} A^s$ over $BG$ is identified with 
$$M_\alpha=\{f\in C^b(EG,A^s);\; f(g\cdot x)=\alpha_g(f(x)),\; \forall x\in EG,\;\forall g\in G\},$$
which is a locally trivial continuous field of $A^s$ over $BG$. 

\begin{lemma} For two $G$-actions $\alpha$ and $\beta$ on $A^s$, we have 
$$[\cP_\alpha]+[\cP_\beta]=[\cP_{\alpha\otimes \beta}]$$ 
in $E_A^1(BG)$. 
\end{lemma}

\begin{proof} It suffices to show $M_\alpha\otimes _{C(BG)}{M_\beta}=M_{\alpha\otimes \beta}$. 
The left-hand side is a subalgebra of the right-hand side, and equality can be shown by using a partition of unity. 
\end{proof}

Let $\alpha$ and $\beta$ be $G$-actions on $A^s$ via $\Aut_0(A^s)$. 
Since $\bE_A^1(BG)$ can be regarded as a subgroup of $E_A^1(BG)$, we also have   
$$[\cP_\alpha^0]+[\cP_\beta^0]=[\cP_{\alpha\otimes \beta}^0]$$ 
in $\bE_A^1(BG)$. 

The author gave a conjecture in \cite{I10} whose special case says that the map $[\alpha]\mapsto[\cP_\alpha]$ gives a bijection between 
the set of cocycle conjugacy classes of outer $G$-actions of $A^s$ and $E_A^1(BG)$. 
After partial results \cite{IM2010},\cite{IMI},\cite{IMII}, the conjecture was recently solved affirmatively by combining 
Meyer's result \cite[Theorem 3.10]{Me19} for surjectivity and Gabe-Szab\'o's result \cite[Theorem 6.2]{GS} for injectivity.

Since the primary obstruction $\delta_0([\cP_\alpha])$ is identified with the composition of $\alpha$ and the quotient map from 
$\Aut(A^s)$ to $\pi_0(\Aut(A^s))$, we have $[\cP_\alpha]\in \ker\delta_0$ if and only if $\alpha$ is via $\Aut_0(A^s)$.  
Thus every element in $\bE_A^1(BG)$ is given by $\cP_\alpha^0$ with an outer action $\alpha$ via $\Aut_0(A^s)$. 
This gives a surjective semigroup homomorphism from $\cE_A'(G)$ onto $\bE_A^1(BG)$. 

\begin{theorem} Let $A$ be a strongly self-absorbing Kirchberg algebra, and let $G$ be a countable discrete amenable group 
with a finite CW-complex model of the classifying space $BG$. 
Let $\alpha$ and $\beta$ be outer $G$-actions on $A^s$ via $\Aut_0(A^s)$. 
Then the following conditions are equivalent: 
\begin{itemize}
\item[$(1)$] $\alpha$ and $\beta$ are $KK$-trivially cocycle conjugate. 
\item[$(2)$] $\alpha$ and $\beta$ are cocycle conjugate. 
\item[$(3)$] $[\cP_\alpha^0]=[\cP_\beta^0]$ in $\bE_A^1(BG)$. 
\item[$(4)$] $[\cP_\alpha]=[\cP_\beta]$ in $E_A^1(BG)$. 
\end{itemize} 
\end{theorem}

\begin{proof} The implication from (1) to (2) is trivial. 
We already mentioned the equivalence of (2) and (4). 
The equivalence of (3) and (4) follows from \cite[Proposition 2.3]{DP-II}. 

We assume (3) and show (1) now. 
From our definition of $\cP_\alpha^0$ and $\cP_\beta^0$, an isomorphism from $\cP_\alpha^0$ to $\cP_\beta^0$ 
is given by a continuous map $\Phi:EG\to \Aut_0(A^s)$ satisfying 
$\Phi(g\cdot x)=\beta_g\circ \Phi(x)\circ \alpha_g^{-1}$ for all $x\in EG$ and all $g\in G.$ 
Indeed, with such a map, we can define an ismorphism $\cP_\alpha^0\to \cP_\beta^0$ by 
$$[(x,\gamma)]\mapsto [(x,\Phi(x)\circ \gamma)].$$
On the other hand, we can show that every isomorphism is of this form by using local trivialization. 
Pointwise application of $\Phi$ gives an isomorphim from $M_\alpha$ to $M_\beta$. 
Now the proof of \cite[Theorem 3.10]{Me19} shows that this isomorphism give rise to a $KK^G$-equivalence from  
$(A^s,\alpha)$ to $(A^s,\beta)$ whose underlying $KK$-equivalence is $[\id]$, 
and \cite[Theorem 6.2]{GS} shows that $\alpha$ and $\beta$ are $KK$-trivially cocycle conjugate. 
\end{proof}

\begin{cor} Let the assumption on $A$ and $G$ be as above.  
The map $[(\alpha,u)]\mapsto [\cP_{\hhalpha}^0]$ gives a group isomorphism from $\cE_A(G)$ onto $\bE_A^1(BG)$. 
\end{cor}

Note that if $\gamma$ is an outer $G$-action on $A$, then $\cP_{\gamma\otimes \id_\K}^0$ is a trivial bundle as $\Aut(A)$ 
is contractible (see \cite[Theorem 2.3]{DP-I}). 
Thus $[(\gamma,1)]$ is the unit of $\cE_A(G)$. 

We denote by $\delta_1$ the primary obstruction map $\delta_1: \bE_A^1(BG)\to H^3(BG,K_0(A))$  
(this definition is a little different from $\delta_1$ in \cite[Definition 4.6]{DP-I}). 
Theorem \ref{primary} says that $\delta_1([\cP^0_{\hhalpha}])=\kappa^3((\alpha,u))$ for $[(\alpha,u)]\in \cE_A(G)$. 
Note that $\delta_1$ is surjective (see the proof of \cite[Proposition 3.9]{DMP}). 

\begin{cor} Let the assumption on $A$ and $G$ be as above. 
Then $\kappa^3$ induces a surjective group homomorphism from $\cE_A(G)$ onto $H^3(G,K_0(A))$. 
\end{cor}

\begin{remark}\label{DD} Let $\mu\in Z^2(G,\T)$, and let $(\id,\mu)$ be a cocycle action of $G$ on $\C$. 
Althoug it is not common to include $\C$ as a strongly self-absorbing C$^*$-algebra, we can still apply Theorem \ref{primary} 
to this cocycle action. 
Then $V_g$ in Eq.(\ref{sd}) is a projective unitary representation with the cocycle $\mu$, and $\Ad V_g$ gives a $G$-action on $\K$. 
We denote by $C_\mu$ the corresponding locally trivial continuous field of $\K$. 
Theorem \ref{primary} shows that its Dixmier-Douady class is $\kappa^3(\id,\mu)=\partial [\mu]$, 
where $\partial: H^2(G,\T)\to H^3(G,\Z)$ is the connecting map of the cohomology long exact sequence arising from the coefficient short exact sequence 
\begin{equation}\label{ZTR}
0\to\Z\to\R\to\T\to0.
\end{equation}
\end{remark}

\begin{lemma}\label{numcocy} 
Let $\mu\in Z^2(G,\T)$ with $j_*\partial [\mu]=0$ in $H^3(G,K_0(A))$, where $j:\Z\to K_0(A)$ is the inclusion map. 
Then for any $[(\alpha,u)]\in \cE_A(G)$, we have $[(\alpha,u)]=[(\alpha,\mu u)]$. 
\end{lemma}

\begin{proof} If $A=\cO_2$, we have nothing to show. 

We assume $A=M_{\fP^\infty}\otimes \cO_\infty$ with possibly empty $\fP$. 
We denote by $C_{(\alpha,\mu)}$ the continuous field of $A^s$ over $BG$ corresponding to $(\alpha,u)$. 
Then we have 
$$C_{(\alpha,\mu u)}\cong C_{(\alpha,u)}\otimes_{C(BG)}C_\mu.$$
If $A=\cO_\infty$, we have $\partial[\mu]=0$, and the Dixmier-Douady class of $C_\mu$ is trivial. 
Thus $C_{(\alpha,\mu u)}\cong C_{(\alpha,u)}$. 

Assume $\fP\neq \emptyset$. 
Note that the condition $j_*\partial [\mu]=0$ implies that $\partial[\mu]$ comes from an element in $H^2(G,\tK_0(A))$. 
Since $BG$ is a finite CW-complex, we have 
$$H^2(BG,\tK_0(M_{\fP^\infty}\otimes \cO_\infty))=\bigoplus_{p\in \fP}\varinjlim_{m} H^2(BG,\Z_{p^m}).$$
Thus there exist primes $p_1,p_2,\cdots, p_l\in \fP$ and natural numbers $m_1,m_2,\cdots, m_l$ satisfying  
$$p_1^{m_1}p_2^{m_2}\cdots p_l^{m_l}\partial [\mu]=0.$$ 
Now \cite[Theorem 2.11]{DP-II} shows $C_\mu\otimes M_{\fP^\infty}\cong C(BG)\otimes M_{\fP^\infty}$, 
and we get $C_{(\alpha,\mu u)}\cong C_{(\alpha,u)}$, and $[(\alpha,\mu u)]=[(\alpha,u)]$. 
\end{proof}

\subsection{Structure of $\cF_A(G)$}
Recall that $\tob:\cF_A(G)\to H^3(G,K^\#_0(A))$ is a semigroup homomorphism, and 
$f:\cE_A(G)\to \cF_A(G)$ is a forgetful functor map. 
We identify $\cE_A(G)$ with $\bE_A^1(BG)$ and we also write $f:\bE_A^1(BG)\to \cF_A(G)$. 
Recall that $\delta_1:\bE_A^1(BG)\to H^3(BG,K_0(A))$ is the primary obstruction map. 

Our main concern in this subsection is the following conjecture. 

\begin{conjecture}\label{C4} Let $A\in \cD_{pi}$, and let $G$ be a countable discrete amenable group with a finite CW-complex model 
of the classifying space $BG$. 
Then following hold:
\begin{itemize}
\item[$(1)$] $\cF_A(G)$ is a group. 
\item[$(2)$] The following sequence is exact: 
$$
\begin{tikzcd}
0\arrow[r]&\ker\delta_1\arrow[r,"f"]&\cF_A(G) \arrow[r,"\tob"]& H^3(G,K^\#_0(A) )\arrow[r]&0
\end{tikzcd}.
$$
\end{itemize}
\end{conjecture}

We first establish the exactness at $\ker\delta_1$ and $\cF_A(G)$ in (2) in full generality. 

\begin{lemma}\label{kernel} The restriction of $f$ to $\ker \delta_1$ is injective, and 
$$\ker \tob=f(\ker \delta_1).$$ 
\end{lemma}

\begin{proof} If $A=\cO_2$, we have $\cE_{\cO_2}(G)=\{0\}$ and $\tob=\ob$. 
Thus there is nothing to show. 

Assume $A=M_{\fP^\infty}\otimes \cO_\infty$ possibly with empty $\fP$. 
We treat $K_0(A)$ as a subring of $\R$. 
We denote by $j$ the inclusion map $j:\Z\to K_0(A)$, and by $\rho$ the inclusion map $\rho:K_0(A)\to \R$. 

Let $[\alpha]\in \ker \tob$. 
Since $\ob(\alpha)=0$, there exists a lifting $(\talpha,u)$ of $\alpha$ such that $(\talpha,u)$ is a cocycle action. 
Thus ${j_A}_*\kappa^3((\talpha,u))=\tob(\alpha)=0$. 
Lemma \ref{presentation} shows that ${q_{K_0(A)\to \tK_0(A)}}_*\kappa^3((\talpha,u))=0$ and 
$\rho_*\kappa^3((\talpha,u))=0$. 
The former implies that there exists $x\in H^3(G,\Z)$ with $j_*x=\kappa^3((\talpha,u))$,  
and the latter shows $(\rho\circ j)_*x=0$.  
This implies that there exists $\mu\in Z^2(G,\T)$ satisfying $x=\partial [\mu]$. 
Thus $\kappa^3((\talpha,\mu^{-1}u))=0$, and $[\alpha]\in f(\ker \delta_1)$, which shows $\ker\tob\subset f(\ker \delta_1)$. 
The other inclusion follows from ${j_A}_*\kappa^3((\beta,v))=\tob\circ f([(\beta,v)])$. 

Now we show that $f$ restricted to $\ker \delta_1$ is injective. 
Let $[(\alpha^{(1)},u^{(1)})],[(\alpha^{(2)},u^{(2)})]\in \ker\delta_1$, and assume  
$f([(\alpha^{(1)},u^{(1))}])=f([(\alpha^{(2)},u^{(2)})])$. 
Then we may assume $\alpha^{(1)}=\alpha^{(2)}$ and $u^{(2)}(g,h)=\mu(g,h)u^{(1)}(g,h)$ with $\mu\in Z^2(G,\T)$.  
This shows 
$$\kappa^3((\alpha^{(2)},u^{(2)}))=\kappa^3((\alpha^{(1)},u^{(1)}))+j_*\partial[\mu],$$ 
and we get $j_*\partial[\mu]=0$. 
Thus from Lemma \ref{numcocy} we obtain $[(\alpha^{(1)},u^{(1)})]=[(\alpha^{(2)},u^{(2)})]$. 
\end{proof}

A typical example of a group $G$ satisfying our assumption is a poly-$\Z$ group, and we recall its definition here.   
A discrete group $G$ is said to be \textit{poly-$\Z$} if there exists a subnormal series 
$$\{e\}=G_0\leq G_1\leq G_2\leq \cdots \leq G_n=G,$$
such that $G_i/G_{i-1}\cong \Z$ for any $1\leq i\leq n$. 
The number $n$ in the above definition is called the Hirsch length of $G$ and denoted by $h(G)$. 
It does not depend on the choice of the subnormal series as above, and coincides with the cohomological dimension of $G$.

\begin{proposition}\label{zero} Let $A$ be a strongly self-absorbing Kirchberg algebra, and let $G$ be a countable torsion-free 
discrete amenable group. 
Let $\gamma$ be an outer $G$-action on $\cO_\infty$. 
Then 
\begin{itemize}
\item[$(1)$] For any $G$-kernel $\alpha:G\to \Out(A)$, we have $[\alpha\otimes \gamma]=[\alpha]$.  
\item[$(2)$] Assume moreover that $G$ is a poly-$\Z$ group. 
Then for any $G$-kernel $\alpha:G\to \Out(A)$, we have $[\alpha\otimes \id_A\otimes \gamma]=[\alpha]$. 
\end{itemize} 
Here we slightly abuse the notation and we denote by the same symbol $\gamma$ the $G$-kernel induced by $\gamma$. 
\end{proposition}

\begin{proof} 
Note that Sz\'abo's result \cite[Theorem 2.6]{Sz18TAMS} holds for a $G$-kernel without essential change of the proof. 
Therefore our task is to construct an equivariant embedding of relevant $G$-C$^*$-algebras into the central sequence algebra of $A$. 

(1)
Since $\gamma$ is unique up to very strongly cocycle conjugate thanks to Gabe-Sz\'abo's classification theorem 
\cite[Corollary 6.11]{GS}, we may assume that $\gamma$ is a quasi-free $G$-action on $\cO_\infty$ given by $\gamma_g(S_h)=S_{gh}$, 
where $\{S_g\}_{g\in G}$ is the canonical generators of $\cO_\infty$. 
Moreover, it is strongly self-absorbing as an action. 

For a free ultrafilter $\omega\in \beta\N\setminus \N$, we set 
$$A^\omega=\ell^\infty(\N,A)/\{(x_n)\in \ell^\infty(\N,A);\;\lim_{n\to\omega}\|x_n\|=0\},$$
and $A_\omega=A^\omega\cap A'$, which is purely infinite by \cite[Proposition 3.4]{KP}. 
Then $\alpha$ induces an outer action on $A_\omega$, which we still denote by $\alpha$ (see the proof of \cite[Lemma 2]{N00}), 
and there exists a unital equivariant embedding of $(\cO_\infty,\gamma)$ into $(A_\omega,\alpha)$ 
(see the proof of \cite[Theorem 5.1]{GI11}, or alternatively \cite[Corollary 2.10]{Sz18CMP}). 
Thus we get $[\alpha\otimes \gamma]=[\gamma]$.  

(2) All we have to show is to give an equivariant embedding of $(A\otimes \cO_\infty,\id_A\otimes \gamma)$ 
into $(A_\omega,\alpha)$. 
As we have already given an embedding of $(\cO_\infty,\gamma)$, it suffices to embed $A$ into $(A_\omega)^G$, where 
$(A_\omega)^G$ is the fixed point subalgebra of the $G$-action $\alpha$ on $A_\omega$. 
We prove it by induction of the Hirsch length of $G$. 

We see that the statement is correct if the Hirsch length is zero, that is $G=\{e\}$, as there exists a unital embedding of $A$ 
into $A_\omega$.  
Assume that the statement is correct for all poly-$\Z$ group whose Hirsch length is less than or equal to $n-1$. 
Assume that the  Hirsch length of $G$ is $n$ now. 
We choose a normal subgroup $N$ of $G$ and $\xi\in G$ such that $G=N\rtimes \langle \xi \rangle$.  
Note that $B:=(A_\omega)^N$ is a purely infinite (see \cite[Corollary 3.2]{IM2010}), and the restriction of $\alpha_\xi$ to $B$, 
which we denote by $\beta$ for simplicity, is an aperiodic automorphism and it has a Rohlin property \cite[Theorem 3.6]{IM2010}. 
By induction hypothesis, there exists a unital embedding $\varphi:A\to B$.  

The following argument is essentially in \cite{SZ19}. 
We choose another embedding $\psi :A\to B$ such that $\psi(A)$ commutes with $\cup_{n\in \Z}\beta^n(\varphi(A))$. 
For $n\in \N$, we define $f_n:\Z\to\R$ by 
$$f_n(x)=\left\{
\begin{array}{ll}
1-\frac{2}{n}|x| , &\quad |x|\leq \frac{n}{2}, \\
0 , &\quad |x|>\frac{n}{2},
\end{array}
\right.
$$
and set 
$$\Theta^{(0)}_n(a)(x)=\sum_{k\in \Z}f(x+kn)\beta^{x+kn}(\varphi(a)),$$
$$\Theta^{(1)}_n(a)(x)=\sum_{k\in \Z}f(x+\frac{n}{2}+kn)\beta^{x+kn}(\psi(a)).$$
Note that for a given $x\in \Z$, we have $f(x+kn)=0$ except for only one $k\in \Z$, and the same statement is true for 
$f(x+\frac{n}{2}+kn)$ too. 
We have 
$$\Theta^{(0)}_n(1_A)(x)+\Theta^{(1)}_n(1_A)(x)=1_B,$$
and the following estimate holds: 
$$\|\Theta^{(j)}_n(a)(x+1)-\beta(\Theta^{(j)}_n(a)(x))\|\leq \frac{2}{n}\|a\|.$$
By construction $\Theta^{(j)}_n(a)$ has period $n$, and it gives rise to an order 0 completely positive map from $A$ to 
$C(\Z/n\Z)\otimes B$. 

We choose a Rohlin tower $\{e^{(0)}_x\}_{x=0}^{n-1}\cup \{e^{(1)}_x\}_{x=0}^n$ commuting with 
$$\bigcup_{n\in \Z}\beta^n(\varphi(A))\cup \bigcup_{n\in \Z}\beta^n(\psi(A)),$$
and set 
$$\Phi^{(j)}_n(a)=\sum_{x=0}^{n-1}e^{(0)}_x\Theta^{(j)}_n(a)(x)+\sum_{x=0}^ne^{(1)}_x\Theta^{(j)}_{n+1}(a)(x).$$
Then $\Phi^{(0)}_n$ and $\Phi^{(1)}_n$ are order 0 completely postive maps from $A$ to $B$ with mutually commuting ranges, 
and they satisfy $\Phi^{(0)}_n(1_A)+\Phi^{(1)}_n(1_A)=1_B$. 
Moreover, we have 
$$\|\Phi^{(j)}_n(a)-\beta(\Phi^{(j)}_n(a))\|\leq \frac{2}{n}\|a\|.$$

Let 
$$\cE(A,A)=\{f\in C([0,1],A\otimes A);\; f(0)\in A\otimes \C,\; f(1)\in \C\otimes A\}.$$
Then there exists a unital homomorphism $\Psi_n:\cE(A,A)\to B$ satisfying 
$\Phi^{(0)}_n(a)=\Psi_n((1-t)(a\otimes 1))$ and $\Phi^{(1)}_n(a)=\Psi_n(t(1\otimes a))$ for all $a\in A$, 
where $t$ is the coordinate function of $[0,1]$ (see \cite[Proposition 3.2]{SZ19}, \cite[Lemma 6.6]{HSWW}).  
Since 
$$\lim_{n\to 0}\|\beta\circ \Psi_n(x)-\Psi_n(x)\|=0$$
holds for all $x$ in a generating set, it holds for general $x\in \cE(A,A)$ too. 
Since $A$ is strongly self-absorbing, there exists an embedding $A\to \cE(A,A)$. 
Thus the usual diagonal sequence argument shows that we have an embedding of $A$ into $B^\beta=(A_\omega)^G$, 
and the induction argument is finished. 
\end{proof}

Note that $\id_A\otimes \gamma$ is a unique outer $G$-action on $A\otimes \cO_\infty\cong A$ up to cocycle conjugacy, 
and it gives a unit of $\cF_A(G)$ if the assumption of Proposition \ref{zero} is satisfied.  

\begin{theorem} Let $G$ be a countably infinite amenable discrete group such that a compact model of 
the classifying space $BG$ exists.
Then 
\begin{itemize}
\item [$(1)$] $\cF_{\cO_\infty}(G)$ is a group. 
\item [$(2)$] Assume moreover that $G$ is poly-$\Z$. 
Then $\cF_{A}(G)$ is a group for $A=M_{\fP^\infty}\otimes \cO_\infty$ and $A=\cO_2$. 
\end{itemize}
\end{theorem}

\begin{proof} Let $A\in \cD_{pi}$. 
For a pair $(A,G)$ satisfying the assumption of (1) or (2), we already know from Proposition \ref{zero} that 
$\cF_A(G)$ has a unit element. 
Thus it suffices to show that each element has its inverse. 

Let $\alpha:G\to \Out(A)$ be a $G$-kernel, and let $(\talpha,u)$ be a lifting of it. 
Note that $A$ is isomorphic to its opposite algebra $A^{\op}$.  
Let $J:A\to A^\op$ be the canonical bijection. 
Then $J\circ \talpha_g\circ J^{-1}$ induces the opposite $G$-kernel $\alpha^\op:G\to \Out(A^\op)$, and 
$$(\talpha_g\otimes J\circ \talpha_g\circ J^{-1}, u(g,h)\otimes J(u(g,h))^*),$$
is a cocycle action. 
Thus there exists its inverse $[(\beta,v)]\in \cE_A(G)$. 
This means that $[\alpha^\op\otimes \beta]$ is the inverse of $[\alpha]$. 
\end{proof}

\begin{cor} If $G$ is poly-$\Z$, 
$$\ob:\cF_{\cO_2}(G)\to H^3(G,\T)$$
is an isomorphism. 
\end{cor}

Our final goal in this section is the following theorem. 

\begin{theorem}\label{main} Let $A=M_{\fP^\infty}\otimes \cO_\infty$ with possibly empty $\fP$.  
Then Conjecture \ref{C4} holds in the following two cases. 
\begin{itemize}
\item[$(1)$] $G=\Z^n$. 
\item[$(2)$] $G$ is a poly-$\Z$ group with $h(G)\leq 5$, where $h(G)$ is the Hirsch length of $G$. 
\end{itemize}
\end{theorem}

We need some preparation for the proof.

Let $\partial: H^2(G,\T)\to H^3(G,\Z)$ and $\partial_A:H^2(G,\T)\to H^3(G,K_0(A))$ be the connecting maps 
of the cohomology long exact sequences arising from Eq.(\ref{ZTR}) and Eq.(\ref{exse}) respectively. 
Then direct computation using Lemma \ref{presentation} shows $\partial_A=j_*\circ \partial$, where 
$j:\Z\to K_0(A)$ is the inclusion map.  

We choose a normal subgroup $N$ of $G$ and $\xi\in G$ such that $G=N\rtimes \langle \xi\rangle$.  
We denote $\xi(n)=\xi n\xi^{-1}$ for $n\in N$. 
We fix an outer $G$-action $\gamma$ on $A$.

\begin{lemma} For $\mu\in Z^2(G,\T)$ satisfying $\partial_A[\mu]=0$,  
there exists a $G$-kernel $\alpha:G\to \Out(A)$ with a lifting $(\talpha,u)$ such that $u(n_1,n_2)=1$ and 
$$u(\xi,n_1)\talpha_{n_1}(u(\xi,n_2))=\mu(\xi(n_1),\xi(n_2))u(\xi, n_1n_2)$$
for all $n_1,n_2\in N$. 
\end{lemma}

\begin{proof} We define an $N$-action $\beta$ on $A$ by $\beta_n=\gamma_{\xi^{-1}(n)}$. 
Note that $\beta$ and $\gamma|N$ are cocycle conjugate. 
Since $j_*\partial[\mu]=0$, Lemma \ref{numcocy} implies $[(\gamma|_N,1)]]=[(\beta,\mu)]$ in $\cE_A(N)$, 
and there exist $\theta\in \Aut(A)$ and $v_n\in U(A)$ satisfying 
$$\theta\circ \beta_n\circ \theta^{-1}=\Ad v_n\circ \gamma_n,$$
$$v_{n_1}\beta_{n_1}(v_{n_2})=\mu(n_1,n_2)v_{n_1n_2}.$$
Let $\talpha_{n\xi^l}=\gamma_n\theta^l\otimes \gamma_{n\xi^l}$ for $n\in N$ and $l\in \Z$, 
which gives rise to a $G$-kernel $\alpha:G\to \Out(A)$. 
Since $\talpha_{n_1}\circ \talpha_{n_2}=\talpha_{n_1n_2}$ for $n_1,n_2\in N$, and 
$\alpha_\xi\circ \talpha_{n}=\Ad (v_{\xi(n)}\otimes 1) \circ \talpha_{\xi(n)\xi}$ for $n\in N$, 
we can choose a representative $(\talpha,u)$ of $\alpha$ with $u(n_1,n_2)=1$ and $u(\xi,n)=v_{\xi(n)}\otimes 1$. 
They have the desired property. 
\end{proof}

We recall an exact sequence arising from the Lyndon/Hochschild-Serr spectral sequence \cite[Chapter VII, Section 6]{B82}: 
$$\begin{tikzcd}
0\arrow[r]&H^1(\Z,H^2(N,M))\arrow[r,"\iota_M"]&H^3(G,M) \arrow[r,"\mathrm{res}"]& H^3(N,M)^\Z\arrow[r]&0
\end{tikzcd}.$$
The map ``$\mathrm{res}$" is the restriction map from $G$ to $N$, and $\Z$ is generated by $\xi$. 
We recall a detailed description of the injective map $\iota_M:H^1(\Z,H^2(N,M))\to H^3(G,M)$ from \cite[Section 7]{IMI}. 
For simplicity, we assume that $M$ is a trivial $G$-module. 

The $\xi$-action on $Z^2(N,M)$ is given by $\xi\cdot\mu(n_1,n_2)=\mu(\xi^{-1}\cdot n_1,\xi^{-1}\cdot n_2)$, 
and 
$$H^1(\Z,H^2(N,M))=H^2(N,M)/(1-\xi_*)H^2(N,M).$$  
We denote by $Q_M$ the quotient map from $H^2(N,M)$ onto $H^1(\Z,H^2(N,M))$. 
Let $\rho\in Z^1(\Z,H^2(N,M))$. 
Then $\rho_m$ is determined by $\rho_1=[\mu]$ and the cocycle relation $\rho_{m+n}=\rho_m+\xi^n_*\rho_n$. 
Now $\iota_M[\rho]\in H^3(G,M)$ is the cohomology class of the following cocycle: 
$$\omega(n_1\xi^{l_1},n_2\xi^{l_2},n_3\xi^{l_3})=-\rho_{l_1}(\xi^{l_1}(n_2),\xi^{l_1+l_2}(n_3)).$$

\begin{lemma} Let $\alpha$ be the $G$-kernel constructed in the previous lemma. 
Then 
$$\ob(\alpha)=\iota_\T\circ Q_\T([\mu]).$$
\end{lemma}

\begin{proof} Throughout this computation, we denote $\talpha$ simply by $\alpha$. 
By construction in the proof of the previous lemma, we may assume  
\begin{equation}\label{un}
u(n_1\xi^{l_1},n_2\xi^{l_2})=\alpha_{n_1}(u(\xi^{l_1},n_2)).
\end{equation}
Let $w_n=v_n\otimes 1$, which satisfies 
$$\alpha_\xi\circ \alpha_{\xi^{-1}(n)}\circ \alpha_\xi^{-1}=\Ad w_n\circ \alpha_n,$$
$$w_{n_1}\alpha_{n_2}(w_{n_2})=\mu(n_1,n_2)w_{n_1n_2}.$$
We can put 
$$u(\xi^l,n)=\left\{
\begin{array}{ll}
\alpha_\xi^{l-1}(w_{\xi(n)})\alpha_\xi^{l-2}(w_{\xi^2(n)})\cdots w_{\xi^l(n)}, &\quad  l>0\\
\alpha_\xi^l(w_n^*)\alpha_\xi^{l+1}(w_{\xi^{-1}(n)}^*)\cdots \alpha_\xi^{-1}(w_{\xi^{l+1}(n)}^*) , &\quad  l<0
\end{array}
\right.
$$
Let $\rho\in Z^1(\Z,Z^2(N,\T))$ be a cocycle determined by $\rho_1=\mu$. 
Then by induction of $l$, we can show 
$$u(\xi^l,n_1)\alpha_{\xi^l(n_1)}(u(\xi^l,n_2))=\rho_l(\xi^l(n_1),\xi^l(n_2))u(\xi^l,n_1n_2).$$

Let $\omega\in Z^3(G,\T)$ be as in Eq.(\ref{obs}). 
Note that the condition Eq.(\ref{un}) implis   
$$\omega(n_1\xi^{l_1},n_2\xi^{l_2},n_3\xi^{l_3})=\omega(\xi^{l_1},n_2\xi^{l_2},n_3),$$
and the 3-cocycle relation implies 
\begin{align*}
\lefteqn{\omega(\xi^{l_1},n_2\xi^{l_2},n_3)} \\
 &=\omega(n_2,\xi^{l_2},n_3)^{-1}\omega(\xi^{l_1}n_2,\xi^{l_2},n_3)\omega(\xi^{l_1},n_2,\xi^{l_2}n_3)
 \omega(\xi^{l_1},n_2,\xi^{l_2})^{-1}\\
 &=\omega(\xi^{l_1},\xi^{l_2},n_3)\omega(\xi^{l_1},n_2,\xi^{l_2}(n_3)).
\end{align*}
The first term on the right-hand side is 
\begin{align*}
\lefteqn{\omega(\xi^{l_1},\xi^{l_2},n_3)}\\
&=\alpha_{\xi}^{l_1}(u(\xi^{l_2},n_3))u(\xi^{l_1},\xi^{l_2}(n_3)\xi^{l_2})
u(\xi^{l_1+l_2},n_3)^*u(\xi^{l_1},\xi^{l_2})^* \\ 
 &=\alpha_{\xi}^{l_1}(u(\xi^{l_2},n_3))u(\xi^{l_1},\xi^{l_2}(n_3))
u(\xi^{l_1+l_2},n_3)^* \\
 &=1.
\end{align*}
The second term is 
\begin{align*}
\lefteqn{\omega(\xi^{l_1},n_2,\xi^{l_2}(n_3))}\\
&=\alpha_\xi^{l_1}(u(n_2,\xi^{l_2}(n_3)))u(\xi^{l_1},n_2\xi^{l_2}(n_3))u(\xi^{l_1}n_2,\xi^{l_2}(n_3))^*u(\xi^{l_1},n_2)^*\\
 &= u(\xi^{l_1},n_2\xi^{l_2}(n_3))\alpha_{\xi^{l_1}(n_2)}(u(\xi^{l_1},\xi^{l_2}(n_3))^*)u(\xi^{l_1},n_2)^*\\
 &=\rho_{l_1}(\xi^{l_1}(n_2),\xi^{l_1+l_2}(n_3))^{-1},
\end{align*}
which shows the statement. 
\end{proof}

\begin{proof}[Proof of Theorem \ref{main}] 
It only remains to show the surjectivity of $\tob$. 
We keep using the decomposition $G=N\rtimes \langle \xi\rangle$ as before. 
Then we have the following commutative diagram with exact rows and the midle column: 
$$\begin{tikzcd}
&&H^3(G,K_0(A))\arrow[d,"{j_A}_*"]&&\\
0\arrow[r]&H^1(\Z,H^2(N,M))\arrow[r,"\iota_{M}"] \arrow[d,"{\ev_1}_*"]&H^3(G,M) \arrow[r,"\mathrm{res}"] \arrow[d,"{\ev_1}_*"]
& H^3(N,M)^\Z\arrow[r] \arrow[d,"{\ev_1}_*"]&0\\
0\arrow[r]&H^1(\Z,H^2(N,\T))\arrow[r,"\iota_\T"]&H^3(G,\T) \arrow[r,"\mathrm{res}"]& H^3(N,\T)^\Z\arrow[r]&0
\end{tikzcd}$$
where $M=K_0^\#(A)$. 

We first show $\Ima \iota_{K^\#_0(A)}\subset \Ima \tob$. 
Let $[\mu]\in H^2(N,K^\#_0(A))$. 
Then since $\partial_A{\ev_1}_*[\mu]=0$, the previous lemma shows that there exists a $G$-kernel $\alpha$ satisfying 
$$\ob(\alpha)=\iota_\T\circ Q_\T({\ev_1}_*([\mu]))={\ev_1}_*\circ\iota_{K^\#_0(A)}(Q_{K^\#_0(A)}([\mu])),$$
which shows $\iota_{K^\#_0(A)}(Q_{K^\#_0(A)}([\mu]))-\tob(\alpha)\in \ker {\ev_1}_*$.
Thus there exists an element $x\in H^3(G,K_0(A))$ satisfying $\iota_{K^\#_0(A)}(Q([\mu]))=\tob(\alpha)+{j_A}_*x$. 
Choosing $[(\beta,v)]\in \cE_A(G)$ with $\kappa^3((\beta,v))=x$, we get $\iota_{K^\#_0(A)}(Q([\mu]))=\tob(\alpha\otimes \beta)$. 

Now to prove the theorem, it suffices to show the following statement: for an arbitrary element $x\in H^3(N,K_0^\#(A))^\Z$, 
there exists a $G$-kernel $\alpha$ satisfying $\mathrm{res}(\tob(\alpha))=x$. 
Indeed, let $y$ be an arbitrary element in $H^3(G,K^\#_0(A))$. 
If we can take $\alpha$ as above for $x=\mathrm{res}(y)$, we get $y-\tob(\alpha)\in \Ima \iota_{K^\#_0(A)}\subset \Ima \tob$, 
and $y\in \Ima \tob$. 

We try to prove the statement by induction of the Hirsch length $h(G)$. 
It works in the case of $G=\Z^n$, while it stops at $h(G)=5$ in the general case. 
If $h(G)\leq 3$, we have $H^3(N,K^\#_0(A))=0$, and certainly the statement holds. 

We first treat the case $G=\Z^n$. 
Assume that $\tob$ is a surjection for $N=\Z^{n-1}$. 
Let $x\in H^3(N,K^\#_0(A))=H^3(N,K^\#_0(A))^\Z$. 
Then by induction hypothesis, there exists an $N$-kernel $\beta:N\to \Out(A)$ satisfying $\tob(\beta)=x$. 
Let $\alpha_{n\xi^l}=\beta_n\otimes \gamma _{n\xi^l}$ for $n\in N$ and $l\in \Z$. 
Then $\mathrm{res}(\tob(\alpha))=x$. 

Now we treat the general case. 
Assume that $\tob$ is surjection for $N$ and assume $h(N)<5$. 
Let $x\in H^3(N,K^\#_0(A))^\Z$. 
Then there exists a $N$-kernel $\beta:N\to \Out(A)$ with a lifting $(\tbeta,v)$ satisfying $\tob(\beta)=x$. 
Note that we also have $\tob(\beta_{\xi(\cdot)})=x$. 
Since $\dim BG\leq 4$, the Atiyah-Hirzebruch spectral sequence \cite[page 175]{DP-I} implies that 
$\delta_1:\bar{E}^1_A(BG)\to H^3(G,K_0(A))$ is an isomorphism, and $\ker\delta_1=\{0\}$. 
Thus Lemma \ref{kernel} implies $[\beta]=[\beta_{\xi\cdot(\cdot)}]$, and there exists $\theta\in \Aut(A)$ 
satisfying $[\theta\circ \tbeta_n\circ\theta^{-1}]=\beta_{\xi(n)}$. 
Letting $\alpha_{n\xi^l}=\beta_n\theta^\ell\otimes \gamma_{n\xi^l}$, we get a $G$-kernel $\alpha:G\to \Aut(A)$ with  
$\mathrm{res}(\tob(\alpha))=x$. 
\end{proof}

\begin{example}
The simplest non-trivial case is $G=\Z^3$. 
In this case, we have $H^3(\Z^3,M)=M$, and $\bar{E}^1_{\cO_\infty}(\T^3)\cong H^3(\T^3,\Z)=\Z$.  
Hence we have the following commutative diagram, 
$$
\begin{tikzcd}
\bar{E}_{\cO_\infty}^1(\T^3) \arrow[r,"f"] \arrow[d,"\simeq"]&H^3(\Z^3,K_0^\#(\cO_\infty)) \arrow[d,"\simeq"]\\
\Z\arrow[r,"j "]&\R
\end{tikzcd},
$$
which shows that the space $H^3(\Z^3,K_0^\#(\cO_\infty))$ of our new invariant naturally interpolates the set of 
Dixmier-Douady classes $H^3(\T^3,\Z)$. 
\end{example}
\section{Comments}
1. To establish Conjecture \ref{C4},(2) for the general poly-$\Z$ case, it is probably better to work on 
$\cO_\infty$ and $M_{\fP^\infty}$ separately instead of $M_{\fP^\infty}\otimes \cO_\infty$ in view of Corollary \ref{pi-stf}. 
In order to be able to push the induction argument in the $\cO_\infty$ case further, we need, as an induction hypothesis, 
a natural splitting of the exact sequence 
$$
\begin{tikzcd}
0\arrow[r]&\ker\delta_1\arrow[r,"f"]&\cF_{\cO_\infty}(N) \arrow[r,"\tob"]& H^3(N,\R)\arrow[r]&0
\end{tikzcd}.
$$
Such a statement is out of reach in our brute force method, and probably requires a homotopy theoretical interpretation of 
the grooup $\cF_{\cO_\infty}(G)$.  
Note that a similar exact sequence 
$$
\begin{tikzcd}
0\arrow[r]&\ker\delta_1\arrow[r]&\bar{E}_{\cO_\infty}^1(X) \arrow[r,"\delta_1"]& H^3(X,\Z)\arrow[r]&0
\end{tikzcd}
$$
splits because $H^3(X,\Z)$ is identified with $\bar{E}_\C^1(X)$. 

2. To show Conjecture \ref{C4},(1) in full generality, we might need a $G$-kernel version of 
Gabe-Sz\'abo's classification theorem. 
A recent work of Arano-Kitamura-Kubota's \cite{AKK} may be the first step toward it. 

3. For non-amenable exact groups $G$ satisfying the Haagerup property, it may still be possible to get reasonable theories 
for $\cE_A(G)$ and $\cF_A(G)$ by considering only amenable actions. 
As there exists a characterization of amenability of group actions in terms of central sequences \cite{OS21}, this definition  
extends to $G$-kernels.

4. To define $\cE_A(G)$ and $\cF_A(G)$ in the stably finite case, we need to assume strongly outerness 
for cocycle actions and $G$-kernels. 
Even with this modification, we cannot expect any similar results in the stably finite case as we have already seen in 
Remark \ref{sf}. 
In fact, Matui-Sato \cite{MS14} showed $\cE_A(\Z^2)\cong H^2(\Z^2,\R/K_0(A))$ for $A=M_{\fP^\infty}$ and $A=\cZ$.



\end{document}